\documentclass[11pt]{amsart}

\usepackage{color,geometry,amssymb}
\usepackage{verbatim,graphicx}
\newtheorem{lm}{Lemma}[section]
\newtheorem{teo}[lm]{Theorem}
\newtheorem{coro}[lm]{Corollary}
\newtheorem{prop}[lm]{Proposition}
\theoremstyle{definition}
\newtheorem{oss}[lm]{Remark}

\newtheorem{defi}[lm]{Definition}
\newtheorem*{ack}{Acknowledgements}

\numberwithin{equation}{section}

\title{The fractional Cheeger problem}
\author[Brasco]{L. Brasco}
\address{{\bf L. B.} and {\bf E. P.}, LATP, Aix--Marseille Universit\'e, 13453 Marseille, France}
\email{lorenzo.brasco@univ-amu.fr, enea.parini@latp.univ-mrs.fr}

\author[Lindgren]{E. Lindgren}
\address{{\bf E. L.} Department of Mathematical Sciences, NTNU, 7491 Trondheim, Norway}
\email{erik.lindgren@math.ntnu.no}
\author[Parini]{E. Parini}


\begin{document}

\maketitle
\begin{abstract}
Given an open and bounded set $\Omega\subset\mathbb{R}^N$, we consider the problem of minimizing the ratio between the $s-$perimeter and the $N-$dimensional Lebesgue measure among subsets of $\Omega$. This is the nonlocal version of the well-known {\it Cheeger problem}. We prove various properties of optimal sets for this problem, as well as some equivalent formulations. In addition, the limiting behaviour of some nonlinear and nonlocal eigenvalue problems is investigated, in relation with this optimization problem. The presentation is as self-contained as possible.
\end{abstract}

\tableofcontents

\section{Introduction}

\subsection{Aim and results of the paper} In this paper we introduce and study the {\it nonlocal/fractional Cheeger problem}  in an open and bounded set $\Omega\subset\mathbb{R}^N$. This amounts to finding a set $E\subset \Omega$ such that
\begin{equation}
\label{eq:cheeger}
\frac{P_s(E)}{|E|}=\inf_{A\subset \Omega}\frac{P_s(A)}{|A|}.
\end{equation}
Here $|\cdot|$ stands for the $N-$dimensional Lebesgue measure, $P_s$ for the {\it nonlocal $s$-perimeter},
$$
P_s(A)=\int_{\mathbb{R}^N}\int_{\mathbb{R}^N}\frac{|1_A (x)-1_A (y)|}{|x-y|^{N+s}}\, dx\, dy, \qquad s\in (0,1),
$$
and $1_A$ is the characteristic function of a set $A$.
An {\it $s-$Cheeger set} of $\Omega$ is a set $E$ satisfying \eqref{eq:cheeger}. Accordingly the quantity 
$$
h_s(\Omega)=\frac{P_s(E)}{|E|},
$$ 
is called the {\it $s-$Cheeger constant} of $\Omega$. We point out that recently the study of nonlocal geometric quantities like $P_s$ has received a great impulse, as they arise in the modelization of phase-transitions in presence of nonlocal interaction terms. We refer to the survey \cite{FV} for an updated account on these studies.
\vskip.2cm
Problem \eqref{eq:cheeger} turns out to have many interesting features and appears to be less obvious to understand than its local counterpart, the (usual) Cheeger problem, where a Cheeger set is a set $E$ achieving the infimum
\begin{equation}
\label{usual}
h_1(\Omega)=\inf_{A\subset \Omega}\frac{P(A)}{|A|},
\end{equation}
with $P(A)$ being the distributional perimeter of $A$, i.e. the total variation of the measure $\nabla 1_A$. Problem \eqref{usual} was first introduced by Jeff Cheeger in \cite{C} in the context of Riemannian Geometry, see also \cite{Pa2} for an overview of the problem.
\par
It is well-known that $h_1(\Omega)$ is indeed an optimal Poincar\'e constant, namely
\[
h_1(\Omega)=\inf_{u\in W^{1,1}_0(\Omega)\setminus\{0\}} \frac{\displaystyle\int_\Omega |\nabla u|\, dx}{\displaystyle\int_\Omega |u|\, dx},
\]
and that $h_1(\Omega)$ is the limit of the first eigenvalue of the $p-$Laplacian as $p$ goes to $1$, see \cite[Corollary 6]{KF}.
In the same spirit, in this paper we prove that the $s-$Cheeger constant can be equivalently characterized as the following $W^{s,1}-$eigenvalue (see Theorem \ref{teo:primo1})
$$
\lambda_{1,1}^s(\Omega):=\inf_{u\in W_0^{s,1}(\Omega)\setminus\{0\}}\frac{\displaystyle\int_{\mathbb{R}^N}\int_{\mathbb{R}^N}\frac{|u(x)-u(y)|}{|x-y|^{N+s}}\, dx\, dy}{\displaystyle\int_{\Omega} |u|\, dx},
$$
and that $h_s(\Omega)$ coincides with the limit as $p$ goes to $1$ of the nonlinear and nonlocal eigenvalues $\lambda^s_{1,p}$ (see Theorem \ref{teo:lim}), coming from the eigenvalue problem
\begin{equation}
\label{eq:speigen}
2\,\int_{\mathbb{R}^N}\frac{|u(y)-u(x)|^{p-2}(u(y)-u(x))}{|x-y|^{N+s\,p}}\,dy+\lambda^s_{1,p}(\Omega)\, |u(x)|^{p-2}u(x)\, =0,\quad x\in\Omega,
\end{equation}
which has been first introduced and studied by Lindqvist and the second author in \cite{LL}. 
\par
We remark that both in \eqref{eq:speigen} and in the definition of the $s-$perimeter, the integrals are taken over the whole $\mathbb{R}^N$ and not only over $\Omega$ itself. The reason is twofold: if one only integrates over $\Omega$ then all sets would have $s-$Cheeger constant equal to zero; on the other hand, the problem \eqref{eq:speigen} would not have the appropriate scaling properties.
\par
For the problem \eqref{eq:speigen} we also provide a global $L^\infty$ estimate for the solutions (Theorem \ref{lm:stimetta}) and a Faber-Krahn inequality with identification of equality cases (Theorem \ref{teo:FK}), which were both missing in \cite{LL}.
\par
Using a scaling argument, it is easy to see that $s-$Cheeger sets must touch the boundary of $\Omega$. We are able to prove that, as in the local case, this happens in a $C^1$ fashion at the points where $\partial \Omega$ is regular. Moreover we show that in the interior of $\Omega$ any $s-$Cheeger set is, up to a singular set of dimension $N-2$, a $C^{1,\alpha}$ surface having constant {\it non-local mean curvature} equal to $-h_s(\Omega)$, in the sense that
\begin{equation}
\label{cmc}
\lim_{\delta\to 0^+}\int_{\mathbb{R}^N\setminus B_\delta(x_0)}\frac{1_E(x)-1_{\mathbb{R}^N\setminus E}(x)}{|x-x_0|^{N+s}}\, dx =-h_s(\Omega),
\end{equation}
for $x_0\in \partial E\cap \Omega$. 
\par
Finally, we provide another alternative characterization of the $s-$Cheeger constant, i.e., 
\begin{equation}
\label{inverse}
\frac{1}{h_s(\Omega)}=\min\{\|\varphi\|_{L^\infty(\mathbb{R}^N\times \mathbb{R}^N)}\,:\, R^*_{s,1}(\varphi)=1 \mbox{ in }\Omega\}, 
\end{equation}
where $R^*_{s,1}$ is the adjoint of the following linear and continuous operator 
$$
u\mapsto \frac{u(x)-u(y)}{|x-y|^{N+s}}. 
$$
This is a nonlocal version of the Max Flow Min Cut Theorem, which can be useful to obtain lower bounds on $h_s(\Omega)$. We recall that for the local case this was investigated in \cite{St}, where the following characterization 
\[
\frac{1}{h_1(\Omega)}=\min_{V\in L^\infty(\Omega;\mathbb{R}^N)}\left\{\|V\|_{L^\infty(\Omega)}\, : \, -\mathrm{div\,} V=1 \mbox{ in }\Omega\right\},
\]
was obtained (see also \cite{Gr}).

\subsection{Open problems} We are left with many open questions and problems. Since the nonlocal mean curvature is a quantity that takes into account the global behavior, the property \eqref{cmc}  
can not in general imply a local characterization of the boundary of a Cheeger set. Even for dimension $N=2$ we are not able to provide any finer information about the interior behaviour of $s-$Cheeger sets, apart from the $C^1$ regularity. However, we should mention that even for the usual Cheeger constant $h_1(\Omega)$, explicitly determining or inferring fine properties of the Cheeger sets are difficult tasks. These usually become affordable for $N=2$, when some severe geometric restrictions are imposed on $\Omega$ (see for example \cite{KT,KP}).
\par
A deep difference between the nonlocal case and the usual one is enlightened by the following behaviour: as it is shown in Remark \ref{oss:buchi}, for a sequence of sets $\{E_k\}_{k\in\mathbb{N}}\subset\mathbb{R}^N$ such that
\[
P(E_k)\le C\qquad \mbox{ and } \qquad \lim_{k\to\infty}|E_k|=0,
\]
the $s-$perimeter as well converges to $0$. This implies for example that in general filling a hole does not decrease the $s-$perimeter, at least if the hole is ``large enough'', while of course this is always the case for the usual distributional perimeter. This behaviour is due to the fact that $P_s(E)$ is a sort of interpolation quantity between $P(E)$ and $|E|$ (see Corollary \ref{coro:relazioni}).
\par
Related is the question of uniqueness of $s-$Cheeger sets which also remains open. While Cheeger sets are known to be unique when $\Omega$ is convex, as proved in \cite{AC,CCN}, this is no longer clear in the nonlocal setting.

\subsection{Plan of the paper}
We start with Section \ref{sec:sobolev}, where we precise the functional analytic setting of our problem and we recall some facts about fractional Sobolev spaces that will be needed throughout the whole paper. Then in Section \ref{sec:first} we recall the definition of first eigenvalue $\lambda^s_{1,p}(\Omega)$ from \cite{LL} and prove that the associated first eigenfunctions are bounded, together with the Faber-Krahn inequality. Section \ref{sec:sperimeter} introduces the $s-$perimeter of a set, there we recall some connections between the naturally associated Sobolev space $W^{s,1}$ and the space of $BV$ functions. With Section \ref{sec:constant} we enter the core of the paper: we introduce problem \eqref{eq:cheeger} and prove some first properties. The remaining sections are then devoted to study regularity issues for $s-$Cheeger sets (Section \ref{sec:reg}), the relation between the first eigenvalues $\lambda^s_{1,p}(\Omega)$ and the $s-$Cheeger constant (Section \ref{sec:lim}) and the alternative 
characterization \eqref{inverse} (Section \ref{sec:maxmin}). Three appendices containing some technical results complement the paper.
\begin{ack}
We thank Pierre Bousquet, Peter Lindqvist and Jean Van Schaftingen for some useful discussions. Giampiero Palatucci kindly provided us a preprint version of \cite{FP}. Part of this work has been done during a visit of the second author to Aix-Marseille Universit\'e. The LATP institution and its facilities are kindly acknowledged.
\end{ack}

\section{A glimpse on fractional Sobolev spaces}
\label{sec:sobolev}

Here and throughout the whole paper we will use the notation $B_r(x_0)$ to denote the open ball of $\mathbb{R}^N$ centered at $x_0$ and with radius $r>0$. Moreover, we will denote by $\omega_k$ the measure of the $k-$dimensional ball with unit radius.
\vskip.2cm\noindent
Given $p\in[1,\infty)$ and $s\in(0,1)$, let us denote by 
\begin{equation}
\label{ps}
[u]_{W^{s,p}(\mathbb{R}^N)}=\left(\int_{\mathbb{R}^N}\int_{\mathbb{R}^N} \frac{|u(x)-u(y)|^p}{|x-y|^{N+s\,p}} \, dx\,dy\right)^\frac{1}{p},
\end{equation}
the $(s,p)$ Gagliardo seminorm in $\mathbb{R}^N$ of a measurable function $u$. Given an open and bounded set $\Omega\subset\mathbb{R}^N$, we first observe that we have
\[
[u]_{W^{s,p}(\mathbb{R}^N)}<+\infty,\qquad \mbox{ for every }u\in C^\infty_0(\Omega).
\] 
We then precise the Sobolev space we want to work with. 
\begin{defi}
The space $\widetilde W^{s,p}_0(\Omega)$ is defined as the closure of $C^\infty_0(\Omega)$ with respect to the norm 
\[
u\mapsto [u]_{W^{s,p}(\mathbb{R}^N)}+\|u\|_{L^p(\Omega)}.
\]
This is a Banach space, which is reflexive for $1<p<\infty$.
\end{defi}
In this paper we will deal with variational problems in the limit case $p=1$, where $\widetilde W^{s,1}_0(\Omega)$ is not reflexive. In this case, we will need the following larger Sobolev space.
\begin{defi}
The space $\mathcal{W}^{s,1}_0(\Omega)$ is defined by
\[
\mathcal{W}^{s,1}_0(\Omega)=\left\{u\in L^1(\Omega)\, :\, [u]_{W^{s,1}(\mathbb{R}^N)}<+\infty \mbox{ and } u=0 \mbox{ a.e. in } \mathbb{R}^N\setminus\Omega\right\}.
\]
Of course, we have $\widetilde W^{s,1}_0(\Omega)\subset \mathcal{W}^{s,1}_0(\Omega)$.
\end{defi}
The following approximation result in $\mathcal{W}^{s,1}_0(\Omega)$ is valid under smoothness assumptions on $\Omega$ and will be quite useful in Section \ref{sec:constant} and \ref{sec:lim}. 
\begin{lm}
\label{lm:LS}
Let $\Omega\subset\mathbb{R}^N$ be an open and bounded Lipschitz set and $s\in(0,1)$. For every $u\in \mathcal{W}^{s,1}_0(\Omega)$ there exists a sequence $\{\varphi_n\}_{n\in\mathbb{N}}\subset C^\infty_0(\Omega)$ such that
\begin{equation}
\label{LS}
\lim_{n\to\infty} \|\varphi_n-u\|_{L^1(\Omega)}=0\qquad \mbox{ and }\qquad \lim_{n\to\infty} [\varphi_n]_{W^{s,1}(\mathbb{R}^N)}=[u]_{W^{s,1}(\mathbb{R}^N)}.
\end{equation}
\end{lm}
\begin{proof}
The proof is based on the construction of \cite[Lemma 3.2]{LS}. Indeed, by this result we know that under the standing assumptions on $\Omega$ there exists a family of diffeomorphisms $\Phi_\varepsilon:\mathbb{R}^N\to\mathbb{R}^N$ with inverses $\Psi_\varepsilon$ such that:
\begin{itemize}
\item we have
\[
\lim_{\varepsilon\to 0^+} \|D\Phi_\varepsilon-\mathrm{Id}\|_{L^\infty}+\|\Phi_\varepsilon-\mathrm{Id}\|_{L^\infty}=0\quad \mbox{ and }\quad \lim_{\varepsilon\to 0^+} \|D\Psi_\varepsilon-\mathrm{Id}\|_{L^\infty}+\|\Psi_\varepsilon-\mathrm{Id}\|_{L^\infty}=0;
\]
\item $\Omega_\varepsilon:=\Phi_\varepsilon(\overline\Omega)\Subset\Omega$ for all $\varepsilon\ll1$.
\end{itemize}
We then define the sequence $\varphi_n=(u\circ\Psi_{1/n})\ast\varrho_n$, where $\varrho_n$ is a standard convolution kernel such that $\|\varrho_n\|_{L^1}=1$. By construction $\varphi_n\in C^\infty_0(\Omega)$ and the first property in \eqref{LS} is easily verified. Observe that by Fatou Lemma, this also implies that
\[
\liminf_{n\to\infty} [\varphi_n]_{W^{s,1}(\mathbb{R}^N)}\ge [u]_{W^{s,1}(\mathbb{R}^N)},
\]
then in order to conclude we just need to prove the upper semicontinuity of the seminorm. We first observe that
\begin{equation}
\label{jensen}
[\varphi_n]_{W^{s,1}(\mathbb{R}^N)}=[(u\circ\Psi_{1/n})\ast\varrho_n]_{W^{s,1}(\mathbb{R}^N)}\le [u\circ\Psi_{1/n}]_{W^{s,1}(\mathbb{R}^N)},
\end{equation}
then the latter can be written as
\[
[u\circ\Psi_{1/n}]_{W^{s,1}(\mathbb{R}^N)}=\int_{\mathbb{R}^N}\int_{\mathbb{R}^N} \frac{|u(z)-u(w)|}{|\Phi_{1/n}(z)-\Phi_{1/n}(w)|}\, |J\Phi_{1/n}(z)|\,|J\Phi_{1/n}(w)|\, dz\, dw, 
\]
by a simple change of variables $(z,w)=(\Psi_{1/n}(x),\Psi_{1/n}(y))$, where $J\Phi_{1/n}$ denotes the Jacobian determinant. We now observe that by construction
\[
|\Phi_{1/n}(z)-\Phi_{1/n}(w)|\ge M_1\, |z-w|\qquad \mbox{ and }\qquad |J\Phi_{1/n}(z)|\le M_2,
\]
for some $M_1>0$ and $M_2\ge 1$ independent of $n$. By applying Lebesgue Dominated Convergence Theorem and keeping into account \eqref{jensen}, we can conclude.
\end{proof}
We now prove a Poincar\'e--type inequality for Gagliardo seminorms.
\begin{lm}
\label{lm:poincare}
Let $1\le p<\infty$ and $s\in(0,1)$, $\Omega\subset\mathbb{R}^N$ be an open and bounded set. There holds
\begin{equation}
\label{poincare}
\|u\|^p_{L^p(\Omega)}\le \mathcal{I}_{N,s,p}(\Omega)\, 
[u]^p_{W^{s,p}(\mathbb{R}^N)},\qquad \mbox{ for every }\varphi\in C^\infty_0(\Omega),
\end{equation}
where the geometric quantity $\mathcal{I}_{N,s,p}(\Omega)$ is defined by
\begin{equation}
\label{cazzo}
\mathcal{I}_{N,s,p}(\Omega)=\min \left\{\frac{\mathrm{diam}(\Omega\cup B)^{N+s\,p}}{|B|}\, :\, B\subset \mathbb{R}^N\setminus\Omega \mbox{ is a ball }\right\}.
\end{equation}
\end{lm}
\begin{proof}
Let $u\in C^\infty_0(\Omega)$ and $B_R\subset\mathbb{R}^N\setminus\Omega$, i.e. a ball of radius $R$ contained in the complement of $\Omega$. For every
$x\in\Omega$ and $y\in B_R$ we then have
\[
|u(x)|^p=\frac{|u(x)-u(y)|^p}{|x-y|^{N+s\,p}}\, |x-y|^{N+s\,p},
\]
from which we can infer
\[
\begin{split}
|B_R|\, |u(x)|^p 
\le \sup_{x\in\Omega, y\in B_R}\,|x-y|^{N+s\,p}\,\int_{B_R} \frac{|u(x)-u(y)|^p}{|x-y|^{N+s\,p}}\, dy.
\end{split}
\]
Integrating on $\Omega$ with respect to $x$ we obtain
\[
\int_\Omega  |u|^p \,dx \le \frac{\mathrm{diam}(\Omega\cup B_R)^{N+s\,p}}{|B_R|} \int_\Omega \int_{B_R} \frac{|u(x)-u(y)|^p}{|x-y|^{N+s\,p}}\, dx\,dy, 
\]
which concludes the proof.
\end{proof}
\begin{oss}
The previous result shows that for an open and bounded set $\Omega\subset\mathbb{R}^N$ the space $\widetilde W^{s,p}_0(\Omega)$ can be equivalently defined as the closure of $C^\infty_0(\Omega)$ with respect to the seminorm $[\,\cdot\,]_{W^{s,p}(\mathbb{R}^N)}$.
\end{oss}
In view of the previous remark, in what follows we will always consider the space $\widetilde W^{s,p}_0(\Omega)$ as equipped with the equivalent norm
\begin{equation}
\label{normetta}
\|u\|_{\widetilde W^{s,p}_0(\Omega)}:=[u]_{W^{s,p}(\mathbb{R}^N)},\qquad u\in \widetilde W^{s,p}_0(\Omega).
\end{equation}
We will also define the space $W^{s,p}_0(\mathbb{R}^N)$ as the closure of $C^\infty_0(\mathbb{R}^N)$ with respect to the norm $[\,\cdot\,]_{W^{s,p}(\mathbb{R}^N)}$.
Then it is immediate to see that the application 
$$
i:\widetilde W^{s,p}_0(\Omega)\to W^{s,p}_0(\mathbb{R}^N),
$$ 
which associates to each $u\in \widetilde W^{s,p}_0(\Omega)$ its extension by $0$ to the whole $\mathbb{R}^N$ is well-defined and continuous. 
\vskip.2cm\noindent
Next, we investigate the behaviour of fractional Sobolev spaces under varying $p$.
\begin{lm}
\label{lm:sommabilità}
Let $\Omega\subset\mathbb{R}^N$ be an open and bounded set. Let $1\le q\le p<\infty$ and $s\in(0,1)$, then for every $0<\varepsilon<s$ we have
\[
[u]_{W^{s-\varepsilon,q}(\mathbb{R}^N)}\le \frac{C}{\varepsilon\,(s-\varepsilon)}\, [u]_{W^{s,p}(\mathbb{R}^N)},\qquad \mbox{ for every }u\in C^\infty_0(\Omega),
\]
where $C=C(N,\Omega,s,p,q)>0$.
\end{lm}
\begin{proof}
Let $u\in C^\infty_0(\Omega)$, by a simple change of variables and using the invariance by translations of $L^p$ norms, 
we have
\[
\begin{split}
[u]^q_{W^{s-\varepsilon,q}(\mathbb{R}^N)}&=\int_{\{h\, :\, |h|>1\}}\int_{\mathbb{R}^N} \frac{|u(x+h)-u(x)|^q}{|h|^{N+(s-\varepsilon)\,q}}\, dx\,dh\\
&+\int_{\{h\, :\, |h|\le 1\}}\int_{\mathbb{R}^N} \frac{|u(x+h)-u(x)|^q}{|h|^{N+(s-\varepsilon)\,q}}\, dx\,dh\\
&\le \frac{2^{q-1}\,N\,\omega_N}{(s-\varepsilon)\,q} \int_{\mathbb{R}^N} |u|^q\, dx+\int_{\{|h|\le 1\}}\left(\int_{\mathbb{R}^N} \frac{|u(x+h)-u(x)|^q}{|h|^{s\,q}}\, dx\right)\frac{dh}{|h|^{N-\varepsilon\,q}}. 
\end{split}
\]
We then observe that
\[
\int_{\mathbb{R}^N} |u|^q\, dx=\int_\Omega |u|^q\, dx\le |\Omega|^{1-\frac{q}{p}}\, \left(\int_\Omega |u|^p\, dx\right)^\frac{q}{p},
\]
and for every $|h|\le 1$, since the function $u(x+h)-u(x)$ has compact support\footnote{More precisely, observe that the support of this function is contained in the open and bounded set $\bigcup\limits_{|h|\le 1} (\Omega+h)$.}, we get
\[
\begin{split}
\int_{\mathbb{R}^N} \frac{|u(x+h)-u(x)|^q}{|h|^{s\,q}}\, dx&\le C_{\Omega,p,q}\,\left(\int_{\mathbb{R}^N} \frac{|u(x+h)-u(x)|^p}{|h|^{s\,p}}\, dx\right)^\frac{q}{p}\le C'\, [u]^q_{W^{s,p}(\mathbb{R}^N)},
\end{split}
\]
where in the last inequality we used Lemma \ref{lm:nikolski}. Putting everything together, we have obtained
\[
[u]^q_{W^{s-\varepsilon,q}(\mathbb{R}^N)}\le \frac{2^{q-1}\,N\,\omega_N}{(s-\varepsilon)\,q}\,|\Omega|^{1-\frac{q}{p}}\, \left(\int_{\Omega} |u|^p\, dx\right)^\frac{q}{p}+\frac{C'\, N\,\omega_N}{\varepsilon\,q}\, [u]^q_{W^{s,p}(\mathbb{R}^N)}.
\]
By using Poincar\'e inequality \eqref{poincare} in the previous, we get the conclusion.
\end{proof}
\begin{teo}
\label{teo:compact}
Let $1\le p<\infty$ and $s\in(0,1)$, let $\Omega\subset\mathbb{R}^N$ be an open and bounded set. Let $\{u_n\}_{n\in\mathbb{N}}\subset \widetilde W^{s,p}_0(\Omega)$ be a bounded sequence, i.e.
\begin{equation}
\label{ipotesi}
\sup_{n\in\mathbb{N}}\, \|u_{n}\|_{\widetilde W^{s,p}_0(\Omega)}<+\infty.
\end{equation}
Then there exists a subsequence $\{u_{n_k}\}_{k\in\mathbb{N}}$ converging in $L^p(\Omega)$ to a function $u$. Moreover, if $p>1$ then $u\in \widetilde W^{s,p}_0(\Omega)$, while for $p=1$ we have $u\in \mathcal{W}^{s,1}_0(\Omega)$.
\end{teo}
\begin{proof}
We first observe that the sequence $\{u_n\}_{n\in\mathbb{N}}$ is bounded in $L^p$ as well, thanks to \eqref{ipotesi} and the Poincar\'e inequality \eqref{poincare}. We then extend by zero the functions $u_n$ to the whole $\mathbb{R}^N$ and observe that in order to get the desired conclusion, by the classical Riesz-Fr\'echet-Kolmogorov compactness theorem we only have to check that
\begin{equation}
\label{rfk}
\lim_{|h|\to 0} \left(\sup_{n\in\mathbb{N}} \int_{\mathbb{R}^N} |u_n(x+h)-u(x)|^p\, dx\right)=0.
\end{equation}
By \eqref{ipotesi} and Lemma \ref{lm:nikolski} we get
\[
\begin{split}
\int_{\mathbb{R}^N} |u_n(x+h)-u_n(x)|^p\, dx&=|h|^{s\,p}\,\int_{\mathbb{R}^N} \frac{|u_n(x+h)-u_n(x)|^p}{|h|^{s\, p}}\, dx\le C\, |h|^{s\,p}\, [u]^p_{W^{s,p}(\mathbb{R}^N)}\le \widetilde C\, |h|^{s\, p},
\end{split}
\]
for every $|h|<1$. The previous estimate implies \eqref{rfk} and this gives the desired conclusion. Finally, the last statement is a consequence of the reflexivity of $\widetilde W^{s,p}_0(\Omega)$ for $p>1$.
\end{proof}
More generally we get the compactness of the embedding in $L^q$ spaces for suitable $q$.
\begin{coro}
\label{coro:compact}
For $1\le p<\infty$ and $s\in(0,1)$, let $\Omega\subset\mathbb{R}^N$ be an open and bounded set. Every bounded sequence $\{u_n\}_{n\in\mathbb{N}}\subset\widetilde  W^{s,p}_0(\Omega)$ admits a
subsequence converging in $L^q(\Omega)$ to a function $u$, for every $q\ge 1$ such that 
\[
q<p^*:=\left\{\begin{array}{cc}
\displaystyle\frac{N\,p}{N-s\,p},& \mbox{ if }s\,p<N,\\
+\infty,& \mbox{ if } s\,p\ge N.
\end{array}
\right.
\]
\end{coro}
\begin{proof}
For $1\le q<p$, we can use Theorem \ref{teo:compact} in conjunction with
\[
\|u\|_{L^q(\Omega)}\le |\Omega|^{\frac{1}{q}-\frac{1}{p}}\, \|u\|_{L^p(\Omega)}.
\]
For $p<q$ it is sufficient to combine the standard interpolation inequality (suppose for simplicity that $s\,p<N$)
\[
\|u\|_{L^q(\Omega)}\le \|u\|^\vartheta_{L^{p^*}(\Omega)}\, \|u\|^{1-\vartheta}_{L^p(\Omega)},\qquad \mbox{ with }\vartheta=\frac{p^*}{q}\frac{q-p}{p^*-p},
\]
and the Sobolev inequality in $W^{s,p}_0(\mathbb{R}^N)$ (see \cite{FS})
with Theorem \ref{teo:compact}.
\end{proof}
For completeness, we conclude this section by considering the case $s\, p>N$. The proof is the same as in \cite[Theorem 8.2]{DPV}, the only difference is that here we work with the narrower space $\widetilde W^{s,p}_0(\Omega)$, so boundary issues can be disregarded. 
\begin{prop}
\label{teo:holderian}
Let $\Omega\subset\mathbb{R}^N$ be an open and bounded set. Let $1< p<\infty$ and $s\in(0,1)$ be such that $s\, p>N$. Then for every $u\in \widetilde W^{s,p}_0(\Omega)$ there holds $u\in C^{0,\alpha}$ with $\alpha=s-N/p$. Moreover there exists a constant $\gamma_{N,s,p}>0$ such that we have the estimates
\begin{equation}
\label{holderian}
|u(x)-u(y)|\le \left(\gamma_{N,s,p}\, \|u\|_{\widetilde W^{s,p}_0(\Omega)}\right)\, |x-y|^\alpha,\qquad x,y\in\mathbb{R}^N,
\end{equation}
and
\begin{equation}
\label{limitata}
\|u\|_{L^\infty}\le \gamma_{N,s,p}\, \|u\|_{\widetilde W^{s,p}_0(\Omega)}\, \mathrm{diam}(\Omega)^{\alpha}.
\end{equation}
\end{prop}
\begin{proof}
By extending $u$ by $0$ to the whole $\mathbb{R}^N$, we can consider it as an element of $W^{s,p}_0(\mathbb{R}^N)$, then we take $x_0\in\mathbb{R}^N$, $\delta>0$ and estimate
\[
\int_{B_\delta(x_0)} |u(x)-\overline u_{x_0,\delta}|^p\, dx\le \frac{1}{|B_\delta(x_0)|}\, \int_{B_\delta(x_0)} \int_{B_\delta(x_0)} |u(x)-u(y)|^p\, dx\,dy,
\]
where $\overline u_{x_0,\delta}$ denotes the average of $u$ on $B_\delta(x_0)$. By observing that $|x-y|\le 2\, \delta$ for every $x,y\in B_\delta(x_0)$ and using that $B_\delta(x_0)=\omega_N\, \delta^N$, we get
\[
\int_{B_\delta(x_0)} |u(x)-\overline u_{x_0,\delta}|^p\, dx\le C\, \delta^{s\,p} \, [u]^p_{W^{s,p}(\mathbb{R}^N)},
\]
that is
\begin{equation}
\label{campanato}
|B_\delta(x_0)|^{-\frac{s\,p}{N}}\, \int_{B_\delta(x_0)} |u(x)-\overline u_{x_0,\delta}|^p\, dx\le C\, [u]^p_{W^{s,p}(\mathbb{R}^N)},
\end{equation}
possibly with a different constant $C>0$. The estimate \eqref{campanato} implies that $u$ belongs to the Campanato space $\mathcal{L}^{p,sp}$, which is isomorphic to $C^{0,\alpha}$ with $\alpha=s-p/N$ (see \cite[Theorem 2.9]{Gi}). This gives \eqref{holderian}, while \eqref{limitata} can be obtained from the previous by simply taking $y$ outside the support of $u$.
\end{proof}

\section{The first fractional eigenvalue}
\label{sec:first}

\begin{defi}
Let $1<p<\infty$ and $s\in(0,1)$. Given an open and bounded set $\Omega\subset\mathbb{R}^N$ we define
\begin{equation}
\label{p}
\lambda^s_{1,p}(\Omega)=\min_{u\in\widetilde W^{s,p}_0(\Omega)} \left\{\|u\|^p_{\widetilde W^{s,p}_0(\Omega)}\, :\,\|u\|_{L^p(\Omega)}=1,\, u\ge 0\right\},
\end{equation}
where the norm $\|\,\cdot\,\|_{\widetilde W^{s,p}_0(\Omega)}$ is defined in \eqref{normetta}.
\end{defi} 
Observe that the constraint $u\ge 0$ in \eqref{p} has no bearing: by dropping it, the minimal value $\lambda^s_{1,p}(\Omega)$ is unchanged, as for every $u\in L^p(\Omega)$ we have
\[
\big||u(x)|-|u(y)|\big|^p\le |u(x)-u(y)|^p\qquad \mbox{ and }\qquad \big\||u|\big\|_{L^p(\Omega)}=\|u\|_{L^p(\Omega)}.
\]
The minimum in \eqref{p} is well-defined thanks to Theorem \ref{teo:compact} (see also \cite[Theorem 5]{LL}) and 
every minimizer $u_\Omega\in \widetilde W^{s,p}_0(\Omega)$ satisfies the following nonlocal and nonlinear equation
\begin{equation}
\label{eigeneq}
\begin{split}
\int_{\mathbb{R}^N}\int_{\mathbb{R}^N} &\frac{|u(x)-u(y)|^{p-2}\, (u(x)-u(y))}{|x-y|^{N+s\,p}}\, (\varphi(x)-\varphi(y))\, dx\, dy=\lambda\, \int_\Omega |u|^{p-2}\,u\, \varphi\, dx,\end{split}
\end{equation}
for every $\varphi\in \widetilde W^{s,p}_0(\Omega)$, with $\lambda=\lambda^s_{1,p}(\Omega)$. 
\begin{oss}
Observe that $\lambda^s_{1,p}(\Omega)$ equals the inverse of the best constant in the Poincar\'e inequality \eqref{poincare}, thus $\lambda^s_{1,p}(\Omega)>0$ thanks to Lemma \ref{poincare}, indeed we have the lower bound
\[
\lambda^s_{1,p}(\Omega)\ge \frac{1}{\mathcal{I}_{N,s,p}(\Omega)},
\] 
with $\mathcal{I}_{N,s,p}$ as in \eqref{cazzo}.
\end{oss}
We show that solutions to \eqref{eigeneq} are globally bounded. The same result can be found in the recent paper \cite{FP}: there a suitable modification of the De Giorgi iteration method is employed. Here on the contrary we use a variant of the Moser iteration technique. We can limit ourselves to prove the result for $s\, p\le N$, since for $s\, p>N$ functions in $\widetilde W^{s,p}_0(\Omega)$ are H\"older continuous and thus bounded, thanks to Proposition \ref{teo:holderian}.

\begin{teo}[Global $L^\infty$ estimate]
\label{lm:stimetta}
Let $1<p<\infty$ and $0<s<1$ such that $s\,p\le N$. 
If $u\in \widetilde W^{s,p}_0(\Omega)$ achieves the minimum \eqref{p}, then $u\in L^\infty(\mathbb{R}^N)$ and for $s\,p<N$ we have the estimate
\begin{equation}
\label{stimettaLinfty}
\|u\|_{L^\infty(\Omega)}\le \widetilde C_{N,p,s}\, \left[\lambda_{1,p}^{s}(\Omega)\right]^{\frac{N}{s\, p^2}}\,\|u\|_{L^p(\Omega)},
\end{equation}
where $\widetilde C_{N,p,s}>0$ is a constant depending only on $N,p$ and $s$ (see Remark \ref{oss:costante} below).
\end{teo}
\begin{proof}
We set for simplicity $\lambda=\lambda^s_{1,p}(\Omega)$ and we first consider the case $s\,p<N$. For every $M$, we define $u_M=\min\{u,M\}$ and observe that $u_M$ is still in $\widetilde W^{s,p}_0(\Omega)$, since this is just the composition of $u$ with a Lipschitz function.
Given $\beta\ge 1$ , we insert the test function $\varphi=u_M^\beta$ in \eqref{eigeneq}, then we get
\[
\int_{\mathbb{R}^N}\int_{\mathbb{R}^N} \frac{|u(x)-u(y)|^{p-2}\, (u(x)-u(y))\, (u^\beta_M(x)-u^\beta_M(y))}{|x-y|^{N+s\, p}}\, dx dy\le \lambda \int_{\mathbb{R}^N} u^{\beta+p-1}\, dx,
\]
where we used that $u_M\le u$.
We now observe that the left-hand side can be estimated from below by a Gagliardo seminorm of some power of $u$. Indeed, by using inequality \eqref{puntuale2} in the Appendix we get
\[
\begin{split}
\int_{\mathbb{R}^N}\int_{\mathbb{R}^N} &\frac{|u(x)-u(y)|^{p-2}\, (u(x)-u(y))}{|x-y|^{N+s\, p}}\, (u^\beta_M(x)-u^\beta_M(y))\, dx dy\\
&\ge \frac{\beta\, p^p}{(\beta+p-1)^p}\, \int_{\mathbb{R}^N}\int_{\mathbb{R}^N} \frac{\left|u^\frac{\beta+p-1}{p}_M(x)-u^\frac{\beta+p-1}{p}_M(y)\right|^p}{|x-y|^{N+s\, p}}\, dxdy.
\end{split}
\]
We can now use the Sobolev inequality for $W^{s,p}_0(\mathbb{R}^N)$, so to get
\[
\int_{\mathbb{R}^N}\int_{\mathbb{R}^N} \frac{\left|u^\frac{\beta+p-1}{p}_M(x)-u^\frac{\beta+p-1}{p}_M(y)\right|^p}{|x-y|^{N+s\, p}}\, dxdy\ge C_{N,p,s}\, \left(\int_{\mathbb{R}^N} \left(u^\frac{\beta+p-1}{p}_M\right)^\frac{N\,p}{N-s\,p}\, dx\right)^\frac{N-s\,p}{N}. 
\]
By keeping everything together and passing to the limit as $M$ goes to $\infty$, we then obtain the following iterative scheme of reverse H\"older inequalities
\begin{equation}
\label{invertito}
\left(\int \left(u^\frac{\beta+p-1}{p}\right)^\frac{N\,p}{N-s\,p}\, dx\right)^\frac{N-s\,p}{N}\le \frac{\lambda}{C_{N,p,s}}\, \left(\frac{\beta+p-1}{p}\right)^{p-1}\, \int \left(u^\frac{\beta+p-1}{p}\right)^p\, dx,
\end{equation}
where we used that $\beta\ge 1$, so that
\[
\frac{\beta+p-1}{p}\, \frac{1}{\beta}\le 1.
\]
Let us now set $\vartheta=\frac{\beta+p-1}{p}$, then the previous inequalities can be written as
\[
\left(\int u^{\vartheta\,\frac{N\,p}{N-s\,p}}\, dx\right)^\frac{N-s\,p}{\vartheta\, N\, p}\le \left(\frac{\lambda}{C_{N,p,s}}\right)^\frac{1}{\vartheta\, p}\, \left(\vartheta^\frac{1}{\vartheta}\right)^\frac{p-1}{p}\, \left(\int u^{\vartheta\, p}\right)^\frac{1}{\vartheta\, p},
\]
that is
\[
\|u\|_{L^{\vartheta\frac{N\, p}{N-s\,p}}}\le \left(\frac{\lambda}{C_{N,p,s}}\right)^\frac{1}{\vartheta\, p}\, \left(\vartheta^\frac{1}{\vartheta}\right)^\frac{p-1}{p}\, \|u\|_{L^{\vartheta\, p}}.
\]
We want to iterate the previous inequality, by taking the following sequence of exponents
\[
\vartheta_0=1,\qquad\qquad \vartheta_{n+1}=\vartheta_n\, \frac{N}{N-s\,p}=\left(\frac{N}{N-s\,p}\right)^{n+1}.
\]
Observe that $N/(N-s\, p)>1$, then $\vartheta_n$ diverges at infinity and in addition
\[
\sum_{n=0}^\infty \frac{1}{\vartheta_n}=\sum_{n=0}^\infty \left(\frac{N-s\,p}{N}\right)^n=\frac{N}{s\,p},
\]
and
\[
\prod_{n=0}^\infty \vartheta_n^\frac{1}{\vartheta_n}=\exp\left(\sum_{n=0}^\infty \frac{n\, \log\left(\frac{N}{N-s\,p}\right)}{\left(\frac{N}{N-s\,p}\right)^n}\right)=\left(\frac{N}{N-s\,p}\right)^{\frac{N-s\, p}{s\,p}\,\frac{N}{s\,p}}.
\]
By starting from $n=0$, at the step $n$ we have
\[
\|u\|_{L^{\vartheta_{n+1}}}\le \left(\left(\frac{\lambda}{C_{N,p,s}}\right)^\frac{1}{p}\right)^{\sum_{i=0}^n \frac{1}{\vartheta_i}}\, \left(\prod_{i=0}^n \vartheta_i^\frac{1}{\vartheta_i}\right)^\frac{p-1}{p}\, \|u\|_{L^p},
\]
then by taking the limit as $n$ goes to $\infty$ we finally obtain
\[
\|u\|_{L^\infty}\le \left(\frac{\lambda}{C_{N,p,s}}\right)^\frac{N}{s\,p^2}\, \left(\frac{N}{N-s\,p}\right)^{\frac{N(N-s\, p)}{s^2\,p^2}\,\frac{p-1}{p}}\,\|u\|_{L^p},
\]
which concludes the proof.
\vskip.2cm\noindent
We now pay attention to the borderline case $s\, p=N$. In this case $\widetilde W^{s,p}_0(\Omega)\hookrightarrow L^q(\Omega)$ for every $q<\infty$. Then we can proceed as before, by replacing Sobolev inequality with the following one
\[
\int_{\mathbb{R}^N}\int_{\mathbb{R}^N} \frac{\left|u^\frac{\beta+p-1}{p}_M(x)-u^\frac{\beta+p-1}{p}_M(y)\right|^p}{|x-y|^{N+s\, p}}\, dxdy\ge \alpha^s_{p}(\Omega)\, \left(\int_{\mathbb{R}^N} \left(u^\frac{\beta+p-1}{p}_M\right)^{2\,p}\, dx\right)^\frac{1}{2},
\]
where 
\[
\alpha^s_{p}(\Omega)=\min_{\widetilde W^{s,p}_0(\Omega)} \left\{[u]^p_{W^{s,p}(\mathbb{R}^N)}\, :\, \|u\|_{L^{2\,p}(\Omega)}=1\right\}.
\]
Then as before we arrive at 
\begin{equation}
\label{invertito2}
\left(\int \left(u^\frac{\beta+p-1}{p}\right)^{2\,p}\, dx\right)^\frac{1}{2}\le \frac{\lambda}{\alpha^s_{p}(\Omega)}\, \left(\frac{\beta+p-1}{p}\right)^{p-1}\, \int \left(u^\frac{\beta+p-1}{p}\right)^p\, dx,
\end{equation}
which is analogous to \eqref{invertito}. By setting again $\vartheta=(\beta+p-1)/p$, we obtain
\[
\left(\int u^{2\,\vartheta\,p}\, dx\right)^\frac{1}{2\,p\,\vartheta}\le \left(\frac{\lambda}{\alpha^s_{p}(\Omega)}\right)^\frac{1}{p\,\vartheta}\, \left(\vartheta^\frac{1}{\vartheta}\right)^\frac{p-1}{p} \left(\int u^{p\,\vartheta}\, dx\right)^\frac{1}{p\,\vartheta}
\]
By iterating the previous with the sequence of exponents
\[
\vartheta_0=1,\qquad \vartheta_{n+1}=2\,\vartheta_n=2^{n+1},
\]
we can conclude the proof as before.
\end{proof}
\begin{oss}
\label{oss:costante}
A closer inspection of the proof informs us that for $s\,p<N$ the constant in \eqref{stimettaLinfty} is given by
\[
\widetilde C_{N,p,s}=\left(\sup_{u\in W^{s,p}_0(\mathbb{R}^N)\setminus\{0\}} \frac{\displaystyle\left(\int_{\mathbb{R}^N} |u|^\frac{N\,p}{N-s\,p}\, dx\right)^\frac{N-s\,p}{N}}{[u]^p_{W^{s,p}(\mathbb{R}^N)}}\right)^\frac{N}{s\,p^2}\, \left(\frac{N}{N-s\,p}\right)^{\frac{N(N-s\, p)}{s^2\,p^2}\,\frac{p-1}{p}}.
\]
The first term is the best constant in the Sobolev inequality for $W^{s,p}_0(\mathbb{R}^N)$, see \cite{FS}. .
\end{oss}
Observe that the quantity $\lambda^s_{1,p}(\Omega)$ enjoys the following scaling law
\[
\lambda^s_{1,p}(t\, \Omega)=t^{-s\,p}\, \lambda^s_{1,p}(\Omega),\qquad t>0,
\]
then the shape functional $\Omega\mapsto |\Omega|^{(s\,p)/N}\, \lambda^s_{1,p}(\Omega)$ is scaling invariant. We have the following.
\begin{teo}[Faber-Krahn inequality]
\label{teo:FK}
Let $1< p<\infty$ and $s\in(0,1)$. For every $\Omega\subset\mathbb{R}^N$ open and bounded, we have
\begin{equation}
\label{FK}
|\Omega|^{(s\,p)/N}\, \lambda^s_{1,p}(\Omega)\ge |B|^{(s\,p)/N}\, \lambda^s_{1,p}(B),
\end{equation}
where $B$ is any $N-$dimensional ball. Moreover, if equality holds in \eqref{FK} then $\Omega$ is a ball. In other words, balls uniquely minimize the first eigenvalue $\lambda^s_{1,p}$ among sets with given $N-$dimensional Lebesgue measure.
\end{teo}
\begin{proof}
Without loss of generality, we can suppose that $|\Omega|=|B|$. Then it is sufficient to use the following {\it P\`olya-Szeg\H{o} principle}
\begin{equation}
\label{PS}
[u]^p_{W^{s,p}(\mathbb{R}^N)}\ge [u^\#]^p_{W^{s,p}(\mathbb{R}^N)},
\end{equation}
which is proved in \cite[Theorem 9.2]{AL}, see also \cite[Theorem A.1]{FS}.
In \eqref{PS} $u^\#$ stands for the symmetric decreasing rearrangement of the function $u$, i.e. $u^\#\in \widetilde W^{s,p}_0(B)$ is the radially symmetric decreasing function such that
\[
|\{x\,:\, u(x)>t\}|=|\{x\, :\, u^\#(x)>t\}|,\qquad t>0.
\]
By using \eqref{PS}, we immediately get \eqref{FK}. For the cases of equality, we observe that if $\lambda^s_{1,p}(\Omega)=\lambda^s_{1,p}(B)$ and $|\Omega|=|B|$, then equality must hold in \eqref{PS}. Again by \cite[Theorem A.1]{FS}, we obtain that any first eigenfunction of $\Omega$ has to coincide with (a translate of) a radially symmetric decreasing function. This implies that $\Omega$ has to be a ball.
\end{proof}

\section{The $s-$perimeter of a set}
\label{sec:sperimeter}
\begin{defi}
For every Borel set $E$, we define its {\it $s-$perimeter} as
\[
P_s(E)=[1_E]_{W^{s,1}(\mathbb{R}^N)}=\int_{\mathbb{R}^N} \int_{\mathbb{R}^N} \frac{|1_E(x)-1_E(y)|}{|x-y|^{N+s}}\, dx\,dy,
\]
where it is understood that $P_s(E)=+\infty$ if the above integral is not finite. 
\end{defi}
Observe that the $s-$perimeter has the following scaling property
\[
P_s(t\, E)=t^{N-s}\, P_s(E),\qquad t>0,
\] 
and we have the isoperimetric inequality
\begin{equation}
\label{siso}
P_s(E)\ge P_s(B)\, \left(\frac{|E|}{|B|}\right)^\frac{N-s}{N},
\end{equation}
where $B$ is any $N-$dimensional ball. Moreover, equality holds in \eqref{siso} if and only if $E$ is a ball, see \cite{FS,FMM}.
It is straightforward to see from the definition that
\[
P_s(E)=2\,\int_{E}\int_{E^c} \frac{1}{|x-y|^{N+s}}\, dx\,dy,
\]
where we set $E^c=\mathbb{R}^N\setminus E$.
In what follows we denote by $BV(\mathbb{R}^N)$ the space
\[
BV(\mathbb{R}^N)=\left\{u\in L^1(\mathbb{R}^N)\, :\, |\nabla u|(\mathbb{R}^N)<+\infty\right\},
\]
where $|\nabla u|(\mathbb{R}^N)$ is the total variation of the distributional gradient of $u$. The following interpolation inequality will be useful.
\begin{prop}
\label{lm:nostro}
Let $s\in(0,1)$. For every $u\in BV(\mathbb{R}^N)$ we have
\begin{equation}
\label{interpolation}
[u]_{W^{s,1}(\mathbb{R}^N)}\le \frac{2^{1-s}\,N\, \omega_N}{(1-s)\, s}\Big[|\nabla u|(\mathbb{R}^N)\Big]^s\,\|u\|^{1-s}_{L^1(\mathbb{R}^N)}.
\end{equation}
\end{prop}
\begin{proof}
Let $u\in BV(\mathbb{R}^N)$, at first we will prove that
\begin{equation}
\label{BVnonscaling}
\int_{\mathbb{R}^N} \int_{\mathbb{R}^N} \frac{|u(x)-u(y)|}{|x-y|^{N+s}}\, dx\, dy\le N\, \omega_N\, \left(\frac{1}{1-s}\, \int_{\mathbb{R}^N} |\nabla u|\, dx+\frac{2}{s}\,\int_{\mathbb{R}^N} |u|\, dx\right).
\end{equation}
We recall that there exists a 
sequence $\{u_n\}_{n\in\mathbb{N}}\subset C^\infty(\mathbb{R}^N)\cap BV(\mathbb{R}^N)$ such that 
\[
\lim_{n\to\infty} \|u_n-u\|_{L^1(\mathbb{R}^N)}=0\qquad \mbox{ and }\qquad \lim_{n\to\infty} \int_{\mathbb{R}^N} |\nabla u_n|\, dx=|\nabla u|(\mathbb{R}^N),
\]
see for example \cite{AFP}.
Then in order to prove \eqref{BVnonscaling} it will be sufficient to prove it for $u_n$.
We have
\[
\begin{split}
\int_{\mathbb{R}^N}\int_{\mathbb{R}^N} \frac{|u_n(x)-u_n(y)|}{|x-y|^{N+s}}\, dx\,dy&=\int_{\mathbb{R}^N}\int_{\mathbb{R}^N} \frac{|u_n(x+h)-u_n(x)|}{|h|^{N+s}}\, dx\,dh\\
&=\int_{\{h\,:\, |h|\ge 1\}}\int_{\mathbb{R}^N} \frac{|u_n(x+h)-u_n(x)|}{|h|^{N+s}}\, dx\,dh\\
&+\int_{\{h\,:\, |h|< 1\}}\int_{\mathbb{R}^N} \frac{|u_n(x+h)-u_n(x)|}{|h|^{N+s}}\, dx\,dh
\end{split}
\]
then we observe that we have
\[
|u_n(x+h)-u_n(x)|\le \left(\int_0^1 |\nabla u_n(x+t\, h)|\, dt\right)\, |h|,\qquad h\in\mathbb{R}^N.
\]
By using this and the invariance of $L^p$ norms by translations, we get
\[
\begin{split}
\int_{\{h\,:\, |h|< 1\}}\int_{\mathbb{R}^N} \frac{|u_n(x+h)-u_n(x)|}{|h|^{N+s}}\, dx\,dh&\le \int_{\{h\, :\, |h|<1\}}\int_{\mathbb{R}^N} \frac{\displaystyle\int_0^1 |\nabla u_n(x+t\, h)|\, dt}{|h|^{N+s-1}} \, dx\,dh\\
&=\int_{\{h\, :\, |h|<1\}}\!\!\frac{\displaystyle\int_{\mathbb{R}^N}|\nabla u_n|\, dx}{|h|^{N+s-1}} \,dh= \frac{N\, \omega_N}{1-s}\, \int_{\mathbb{R}^N} |\nabla u_n|\, dx.
\end{split}
\]
For the other integral, by using the triangular inequality and again the invariance of $L^p$ norms by translations, we get
\[
\begin{split}
\int_{\{h\,:\, |h|\ge 1\}}\int_{\mathbb{R}^N} \frac{|u_n(x+h)-u_n(x)|}{|h|^{N+s}}\, dx\,dh&\le \int_{\{h\, :\, |h|\ge1\}}\,\frac{1}{|h|^{N+s}}\,\left(\int_{\mathbb{R}^N} |u_n(x+h)|\, dx\right)dh\\
&+\left(\int_{\{h\, :\, |h|\ge 1\}} \frac{dh}{|h|^{N+s}}\right)\,\left(\int_{\mathbb{R}^N} |u_n|\, dx\right)\\
&=2\,\frac{N\,\omega_N}{s}\, \int_{\mathbb{R}^N} |u_n|\, dx.\\
\end{split}
\]
In conclusion we obtained \eqref{BVnonscaling} for the sequence $\{u_n\}_{n\in\mathbb{N}}$ and thus for $u$, by passing to the limit.
\vskip.2cm\noindent
In order to arrive at \eqref{interpolation}, it is now sufficient to use a standard homogeneity argument. Let $u\in BV(\mathbb{R}^N)\setminus\{0\}$ and set $u_\lambda(x)=u(x/\lambda)$, where $\lambda>0$. Then by \eqref{BVnonscaling} we get
\[
\lambda^{N-s}\,\int_{\mathbb{R}^N} \int_{\mathbb{R}^N} \frac{|u(x)-u(y)|}{|x-y|^{N+s}}\, dx\, dy\le N\, \omega_N\, \left(\frac{\lambda^{N-1}}{1-s}\, \int_{\mathbb{R}^N} |\nabla u|\, dx+\frac{2\,\lambda^N}{s}\,\int_{\mathbb{R}^N} |u|\, dx\right),
\] 
that is
\begin{equation}
\label{omogenizzo!}
\lambda^{1-s}\,\int_{\mathbb{R}^N} \int_{\mathbb{R}^N} \frac{|u(x)-u(y)|}{|x-y|^{N+s}}\, dx\, dy-2\, N\,\omega_N\, \frac{\lambda}{s}\,\int_{\mathbb{R}^N} |u|\, dx\le \frac{N\,\omega_N}{1-s}\, \int_{\mathbb{R}^N} |\nabla u|\, dx.
\end{equation}
The left-hand side is maximal for 
\[
\lambda=\left(\frac{(1-s)\,s\, [u]_{W^{s,1}(\mathbb{R}^N)}}{2\,N\,\omega_N\, \|u\|_{L^1(\mathbb{R}^N)}}\right)^\frac{1}{s}.
\]
By replacing this value in \eqref{omogenizzo!}, we obtain the desired result.
\end{proof}
\begin{oss}
\label{oss:sciarpa}
We point out that in dimension $N=1$ inequality \eqref{interpolation} is sharp for every $s\in(0,1)$, since equality is attained for characteristic functions of bounded intervals. Let $u=1_I$ be the characteristic function of the interval $I$ having length $\ell$, a direct computation gives
\[
[u]_{W^{s,1}(\mathbb{R}^N)}=P_s(I)=\frac{4\, \ell^{1-s}}{s\,(1-s)},
\] 
while
\[
\omega_1=2,\qquad \|u\|_{L^1(\mathbb{R})}=\ell, \qquad |u'|(\mathbb{R})=2,
\]
then it is easily seen that equality holds in \eqref{interpolation}. 
\end{oss}
We now highlight a couple of consequences of inequality \eqref{interpolation}. The first one gives a relation between the $s-$perimeter and the standard distributional perimeter. For the proof it is sufficient to take $u=1_E$ in \eqref{interpolation}. A related estimate for $N=2$ can be found in \cite[Lemma 2.2]{MP}.
\begin{coro}
\label{coro:relazioni}
Let $s\in(0,1)$, for every finite perimeter set $E\subset\mathbb{R}^N$ we have
\[
P_s(E)\le \frac{2^{1-s}\, N\,\omega_N}{(1-s)\,s}\,P(E)^s\, |E|^{1-s}.
\]
\end{coro}
\begin{oss}
\label{oss:buchi}
The previous result implies that if $\{E_k\}\subset \mathbb{R}^N$ is such that $P(E_k)\le C$ and $|E_k|$ converges to $0$ as $k$ goes to $\infty$, then $P_s(E_k)$ as well converges to $0$. For example, by taking the annular set $C_k=\{x\,: \, 1-1/k<|x|<1 \}$, we get that $P_s(C_k)$ is going to $0$ as $k$ goes to $\infty$. Then in general for the $s-$perimeter it is not true that {\it filling a hole decreases the perimeter}, like in the standard case.
\end{oss}
By simply using Poincar\'e inequality in \eqref{interpolation}, we can also infer the following.
\begin{coro}
\label{coro:nostroerik}
Let $s\in(0,1)$. For every $u\in BV(\mathbb{R}^N)$ with compact support there holds
\[
\int_{\mathbb{R}^N} \int_{\mathbb{R}^N} \frac{|u(x)-u(y)|}{|x-y|^{N+s}}\, dx\, dy\le \frac{2^{1-s}\,N\, \omega_N}{(1-s)\, s}\,\mathrm{diam}(\mathrm{spt\,}u)^{1-s}\,|\nabla u|(\mathbb{R}^N),
\]
where $\mathrm{spt}(u)$ denotes the support of $u$. 
\end{coro}
In what follows, we will need the following Coarea Formula for nonlocal integrals. This has been first proved by Visintin in \cite{Vi}. The proof is omitted, we just recall that it is based on Fubini's Theorem and on the so-called Layer Cake Representation for functions.
\begin{lm}
\label{lm:coarea}
Let $u\in L^1(\mathbb{R}^N)$, then there holds the following Coarea-type formula
\begin{equation}
\label{coarea}
[u]_{W^{s,1}(\mathbb{R}^N)}=\int_{0}^\infty P_s(\{x\, :\, |u(x)|>t\})\, dt.
\end{equation}
In particular, if $[u]_{W^{s,1}(\mathbb{R}^N)}<+\infty$ then for almost every $t\in\mathbb{R}$ the sets $\{x\, :\, |u(x)|>t\}$ has finite $s-$perimeter.
\end{lm}
By Proposition \ref{lm:nostro} and Lemma \ref{lm:coarea}, we can infer the following limiting behaviour for the $(s,1)$ Gagliardo seminorm, whose proof is essentially the same as \cite[Theorem 8]{Lu}. We give it for ease of completeness.
\begin{prop}
\label{prop:conv}
Let $u \in BV(\mathbb{R}^N)$ have compact support. Then there holds
\begin{equation}
\label{conv}
\lim_{s\nearrow 1} (1-s)\,[u]_{W^{s,1}(\mathbb{R}^N)}=2\,\omega_{N-1}\, |\nabla u|(\mathbb{R}^N).
\end{equation}
\end{prop}
\begin{proof}
First of all, we remark that $[u]_{W^{s,1}(\mathbb{R}^N)}<+\infty$ for every $s < 1$, thanks to Proposition \ref{lm:nostro}. By the coarea formula \eqref{coarea}
\[ 
(1-s)\,[u]_{W^{s,1}(\mathbb{R}^N)} = (1-s)\int_0^{+\infty} P_s(\Omega_t)\,dt,
\]
where we set $\Omega_t :=\{|u|>t\}$. Since by definition $P_s(\Omega_t) = [1_{\Omega_t}]_{W^{s,1}(\mathbb{R}^N)}$, by Corollary \ref{coro:nostroerik} we have that\footnote{The constant $C$ only depends on the dimension $N$ and the diameter of $\mathrm{spt\,} u$, for $s$ close to $1$.} 
$$
(1-s)\,P_s(\Omega_t) \leq C\, P(\Omega_t),\qquad t>0,
$$ 
where $P$ denotes the usual distributional perimeter. By using the usual coarea formula for $BV$ functions (see \cite{AFP}), we get 
\[ 
\int_0^{+\infty} P(\Omega_t)\,dt = |\nabla u|(\mathbb{R}^N) < +\infty.
\]
On the other hand, by \cite[Theorem 4]{Lu} we have\footnote{The reader should pay attention to the fact that our definition of $P_s(\Omega)$ differs from that of \cite{Lu} by a multiplicative factor $2$.}
\[ 
\lim_{s \nearrow 1} (1-s)\,P_s(\Omega_t) =2\, \omega_{N-1}\, P(\Omega_t).
\]
We point out that the constant $\omega_{N-1}$ can be deduced from formula (4) in \cite{Lu}.
Therefore it is possible to apply Lebesgue's Dominated Convergence Theorem in order to obtain
\begin{align*} 
\lim_{s \nearrow 1} (1-s)\,[u]_{W^{s,1}(\mathbb{R}^N)} & = \lim_{s \nearrow 1} (1-s)\int_0^{+\infty} P_s(\Omega_t)\,dt  = 2\, \omega_{N-1} \int_0^{+\infty} P(\Omega_t)\,dt =2\, \omega_{N-1}\, |\nabla u|(\mathbb{R}^N),
\end{align*}
thus concluding the proof.
\end{proof}
We also recall the sharp Sobolev inequality in $W^{s,1}_0(\mathbb{R}^N)$, which is nothing but a functional version of the isoperimetric inequality \eqref{siso}. 
\begin{teo}[Sobolev inequality in $W^{s,1}_0(\mathbb{R}^N)$]
Let $N\ge 2$ and $s\in(0,1)$, then
\begin{equation}
\label{isoper}
\min_{u\in W^{s,1}_0(\mathbb{R}^N)\setminus\{0\}} \frac{\displaystyle\int_{\mathbb{R}^N} \int_{\mathbb{R}^N} \frac{|u(x)-u(y)|}{|x-y|^{N+s}}\, dx\,dy}{\displaystyle\left(\int_{\mathbb{R}^N} |u|^\frac{N}{N-s}\, dx\right)^\frac{N-s}{N}}=P_s(B)\, |B|^\frac{s-N}{N},
\end{equation}
where $B$ is any $N-$dimensional ball. The minimum in \eqref{isoper} is attained by any characteristic function of an $N-$dimensional ball.
\end{teo}
\begin{proof}
We at first observe that it is sufficient to prove the result for positive functions.
Let $u\in W^{s,1}_0(\mathbb{R}^N)$ be positive and let us indicate with $\mu$ its distribution function
\[
\mu(t)=|\{x\, : \, u(x)>t\}|.
\]
By using the Cavalieri principle we get the following estimate (see \cite[Section 1.3.3]{Ma})
\[
\left(\int_{\mathbb{R}^N} |u|^\frac{N}{N-s}\, dx\right)^\frac{N-s}{N}\le \int_0^\infty \mu(t)^\frac{N-s}{N}\, dt.
\]
Using the latter, \eqref{coarea} and the isoperimetric inequality \eqref{siso}, we get the estimate
\[
\begin{split}
\frac{\displaystyle\int_{\mathbb{R}^N} \int_{\mathbb{R}^N} \frac{|u(x)-u(y)|}{|x-y|^{N+s}}\, dx\,dy}{\displaystyle\left(\int_{\mathbb{R}^N} |u|^\frac{N}{N-s}\right)^\frac{N-s}{N}}&\ge \frac{\displaystyle\int_{0}^\infty P_s(\{u>t\})\, dt}{\displaystyle\int_0^\infty \mu(t)^\frac{N-s}{N}\, dt}\ge P_s(B)\, |B|^\frac{s-N}{N},
\end{split}
\]
for all $u\in W^{s,1}_0(\mathbb{R}^N)$. On the other hand, by taking $u=1_B$ with $B$ any $N-$dimensional ball, we get equality in the previous. 
\end{proof}

\section{The nonlocal Cheeger constant}
\label{sec:constant}
\begin{defi}
Let $s\in(0,1)$. For every open and bounded set $\Omega\subset\mathbb{R}^N$ we define its {\it $s-$Cheeger constant} by
\begin{equation}
\label{def_ch}
h_s(\Omega)=\inf_{E\subset\Omega} \frac{P_s(E)}{|E|}.
\end{equation}
A set $E_\Omega\subset\Omega$ achieving the infimum in the previous problem is said to be an {\it $s-$Cheeger set} of $\Omega$. Also, we say that $\Omega$ is {\it $s-$calibrable} if it is an $s-$Cheeger set of itself, i.e. if
\[
h_s(\Omega)=\frac{P_s(\Omega)}{|\Omega|}.
\]
\end{defi}
\begin{oss}
As in the local case, any ball $B\subset\mathbb{R}^N$ is $s-$calibrable. This is a direct consequence of the isoperimetric inequality \eqref{siso}, which gives for every $E\subset B$
\[
\frac{P_s(E)}{|E|}\ge \frac{P_s(B)}{|B|}\, \left(\frac{|B|}{|E|}\right)^\frac{s}{N}\ge \frac{P_s(B)}{|B|}.
\]
\end{oss}
\begin{prop}
Let $s\in(0,1)$, every $\Omega\subset\mathbb{R}^N$ open and bounded admits an $s-$Cheeger set. Moreover, if $E_\Omega$ is an $s-$Cheeger set of $\Omega$, then $\partial E_\Omega\cap\partial\Omega\not=\emptyset$.
\end{prop}
\begin{proof}
First of all, we observe that $h_s(\Omega)<+\infty$, i.e. there exists at least an admissible set such that the ratio defining $h_s(\Omega)$ is finite. Indeed, since $\Omega$ is open, it contains a ball $B_r$ and for this $P_s(B_r)<+\infty$.
We then take a minimizing sequence $\{E_n\}_{n\in\mathbb{N}}\subset \Omega$ and we can obviously suppose that
\[
\frac{P_s(E_n)}{|E_n|}\le h_s(\Omega)+1,\qquad \mbox{ for every }n\in\mathbb{N}.
\] 
As $|E_n|\le |\Omega|$, the previous immediately gives a uniform bound on the $s-$perimeter of the sequence $\{E_n\}_{n\in\mathbb{N}}$. Moreover, by combining the previous and \eqref{siso}, we get
\[
|E_n|^\frac{N-s}{N}\, \left(\frac{P_s(B)}{|B|^\frac{N-s}{N}}\right)\le (h_s(\Omega)+1)\,|E_n|,
\]
which in turn implies
\begin{equation}
\label{lowerbound}
|E_n|\ge c_{N,\Omega,s}>0.
\end{equation}
Then we get
\[
[1_{E_n}]_{W^{s,1}(\mathbb{R}^N)}+\|1_{E_n}\|_{L^1}\le C, \qquad \mbox{ for every }n\in\mathbb{N}.
\]
By appealing to Theorem \ref{teo:compact}, this in turn implies that the sequence $\{1_{E_n}\}_{n\in\mathbb{N}}$ is strongly converging (up to a subsequence, not relabeled) in $L^1$ to a function $\varphi$, which has the form $\varphi=1_{E_\Omega}$ for some measurable set $E_\Omega\subset\Omega$. Thanks to \eqref{lowerbound}, we can also assure that $|E_\Omega|>0$. By using the latter and the lower semicontinuity of the Gagliardo seminorms, we get
\[
\frac{[1_{E_\Omega}]_{W^{s,1}(\mathbb{R}^N)}}{|E_\Omega|}\le \liminf_{n\to\infty}\frac{[1_{E_n}]_{W^{s,1}(\mathbb{R}^N)}}{|E_n|}=h_s(\Omega).
\]
This concludes the proof of the existence.
\vskip.2cm\noindent
Let us now prove the second statement. Assume by contradiction that $E_\Omega \Subset \Omega$. Then, for $t > 1$ sufficiently close to $1$, the scaled set $t\,E_\Omega$ is still contained in $\Omega$. We have
\[ 
\frac{P_s(t\,E_\Omega)}{|t\,E_\Omega|}= \frac{t^{N-s}\,P_s(E_\Omega)}{t^N |E_\Omega|} = t^{-s} h_s(\Omega) < h_s(\Omega), 
\]
contradicting the minimality of $E_\Omega$. Hence we obtain the claim.
\end{proof}
\begin{oss}
We have seen that an $s-$Cheeger $E_\Omega$ of $\Omega$ has to touch the boundary $\partial\Omega$. Actually, the previous proof shows that $E_\Omega$ has the following slightly stronger property: $t\, E_\Omega$ is not contained in $\Omega$ for any $t>1$.
\end{oss}
It is not difficult to see that balls (uniquely) minimize the $s-$Cheeger constant among sets having given $N-$dimensional measure. This can be seen as a limit case of the Faber-Krahn inequality \eqref{FK}.
\begin{prop}
Let $s\in(0,1)$, for every open and bounded set $\Omega\subset\mathbb{R}^N$ we have
\begin{equation}
\label{FKcheeger}
|\Omega|^\frac{s}{N}\, h_s(\Omega)\ge |B|^\frac{s}{N}\, h_s(B),
\end{equation}
where $B$ is any $N-$dimensional ball. Equality in \eqref{FKcheeger} holds if and only if $\Omega$ itself is a ball.
\end{prop}
\begin{proof}
Let $B$ be a ball such that $|\Omega|=|B|$ and let $E_\Omega$ be an $s-$Cheeger set for $\Omega$. By using \eqref{siso} we have
\[
h_s(\Omega)=\frac{P_s(E_\Omega)}{|E_\Omega|}\ge \frac{P_s(B)}{|B|}\, \left(\frac{|B|}{|E_\Omega|}\right)^\frac{s}{N}\ge \frac{P_s(B)}{|B|}= h_s(B),
\]
where we used that $|E_\Omega|\le |\Omega|=|B|$. The characterization of equality cases directly follows from the equality cases in \eqref{siso}.
\end{proof}
Thanks to Corollary \ref{coro:relazioni}, the nonlocal quantity $h_s(\Omega)$ can be estimated in terms of the usual (local) Cheeger constant $h_1(\Omega)$. 
\begin{prop}
\label{prop:hsh1}
Let $s\in(0,1)$ and $\Omega\subset\mathbb{R}^N$ be an open bounded set. Then we have
\[
h_s(\Omega)\le \frac{2^{1-s}\, N\,\omega_N}{(1-s)\, s}\, h_1(\Omega)^s.
\]
\end{prop}
\begin{proof}
Let $E\subset \Omega$ be a Cheeger set, then by using Corollary \ref{coro:relazioni} we get
\[
h_1(\Omega)^s=\left(\frac{P(E)}{|E|}\right)^s\ge \frac{(1-s)\,s}{2^{1-s}\,N\,\omega_N}\, \frac{P_s(E)}{|E|},
\]
which gives the conclusion.
\end{proof} 

We now provide an equivalent definition of $h_s(\Omega)$. Let us define
\[
\lambda^s_{1,1}(\Omega)=\inf_{u\in \widetilde W^{s,1}_0(\Omega)} \left\{[u]_{W^{s,1}(\mathbb{R}^N)}\, :\, \|u\|_{L^1(\Omega)}=1,\ u\ge 0\right\}.
\]
This variational problem in general has a ``relaxed'' solution, i.e. this infimum is attained in the larger space $\mathcal{W}^{s,1}_0(\Omega)$, at least for $\Omega$ smooth enough. This is the content of the next result.
\begin{lm}
\label{lm:salvaculo}
Let $s\in(0,1)$ and $\Omega\subset\mathbb{R}^N$ be an open and bounded Lipschitz set. Then \begin{equation}
\label{smoke}
\lambda^s_{1,1}(\Omega)=\min_{u\in \mathcal{W}^{s,1}_0(\Omega)} \left\{[u]_{W^{s,1}(\mathbb{R}^N)}\, :\, \|u\|_{L^1(\Omega)}=1,\ u\ge 0\right\},
\end{equation}
and the minimum on the right is attained.
\end{lm}
\begin{proof}
Of course, since $\widetilde W^{s,1}_0(\Omega)\subset\mathcal{W}^{s,1}_0(\Omega)$, we have
\begin{equation}
\label{right}
\inf_{u\in \mathcal{W}^{s,1}_0(\Omega)} \left\{[u]_{W^{s,1}(\mathbb{R}^N)}\, :\, \|u\|_{L^1(\Omega)}=1,\ u\ge 0\right\}\le \lambda^s_{1,1}(\Omega),
\end{equation}
then we just have to show the reverse inequality. At first, we observe that the infimum in the left-hand side of \eqref{right} is attained by some function $u_0\in \mathcal{W}^{s,1}_0(\Omega)$, again thanks to Theorem \ref{teo:compact}.
Then we observe that since $\Omega$ is Lipschitz, by Lemma \ref{lm:LS} there exists a sequence $\{\varphi_n\}_{n\in\mathbb{N}}\subset C^\infty_0(\Omega)$ such that
\[
\lim_{n\to\infty} \|\varphi_n-u_0\|_{L^1(\Omega)}=0\qquad \mbox{ and } \qquad \lim_{n\to\infty}[\varphi_n]_{W^{s,1}(\mathbb{R}^N)}=[u_0]_{W^{s,1}(\mathbb{R}^N)}.
\]
As $C^\infty_0(\Omega)\subset\widetilde W^{s,1}_0(\Omega)$, by appealing to the definition of $\lambda^s_{1,1}(\Omega)$ we get
\[
\lambda^s_{1,1}(\Omega)\le \lim_{n\to\infty}\frac{[\varphi_n]_{W^{s,1}(\mathbb{R}^N)}}{\|\varphi_n\|_{L^1(\Omega)}}=[u_0]_{W^{s,1}(\mathbb{R}^N)}.
\]
By using the minimality of $u_0$ and \eqref{right}, we get \eqref{smoke}.
\end{proof}
Then the main result of this section is the following characterization of $h_s(\Omega)$.
\begin{teo}
\label{teo:primo1}
Let $s\in(0,1)$ and let $\Omega\subset\mathbb{R}^N$ be an open and bounded set. For every $u\in \mathcal{W}^{s,1}_0(\Omega)\setminus\{0\}$, we have
\begin{equation}
\label{pensaci!}
\frac{[u]_{W^{s,1}(\mathbb{R}^N)}}{\|u\|_{L^1(\Omega)}}\ge h_s(\Omega).
\end{equation} 
Moreover, if equality holds in \eqref{pensaci!}, then $u$ has the following property: almost every level set of $u$ with positive $N-$dimensional Lebesgue measure is an $s-$Cheeger set of $\Omega$.
\par
Finally, if $\Omega$ has Lipschitz boundary then
\begin{equation}
\label{uguali}
\lambda^s_{1,1}(\Omega)=h_s(\Omega).
\end{equation}
\end{teo}
\begin{proof}
The proof of the first part is based on the Coarea Formula of Lemma \ref{lm:coarea}. Without loss of generality, we can suppose that $u$ is positive. Then by \eqref{coarea}, Cavalieri formula and the definition of $h_s(\Omega)$ we get
\[
\frac{[u]_{W^{s,1}(\mathbb{R}^N)}}{\|u\|_{L^1(\Omega)}}=\frac{\displaystyle\int_{0}^\infty P_s(\{x\, :\, u(x)>t\})\, dt}{\displaystyle\int_{0}^\infty |\{x\, :\, u(x)>t\}|\, dt}\ge h_s(\Omega),
\] 
which proves \eqref{pensaci!}. The property of the level sets of an optimal function $u$ is a consequence of the previous estimate, since if equality holds then we must have
\[
P_s(\{x\, :\, u(x)>t\})=h_s(\Omega)\, |\{x\, :\, u(x)>t\}|,
\]
for almost every level $t$.
\vskip.2cm\noindent
In order to prove \eqref{uguali}, we at first observe that the previous estimate easily implies
\[
\lambda^s_{1,1}(\Omega)\ge h_s(\Omega).
\]
On the other hand, by Lemma \ref{lm:salvaculo} we have that the variational problem giving $\lambda^s_{1,1}(\Omega)$ is the same as $h_s(\Omega)$, but in the latter we restricted the competitors to a narrower class. This implies
\[
\lambda^s_{1,1}(\Omega)\le h_s(\Omega),
\]
so that equality \eqref{uguali} holds. 
\end{proof}

\section{Regularity of $s-$Cheeger sets}
\label{sec:reg}
Following \cite{CRS10}, given two sets $A,B\subset\mathbb{R}^N$ and $0<s<1$, we introduce the following notation
$$
L(A,B)=\int_{A}\int_{B}\frac{1}{|x-y|^{N+s}} \,dx\, dy.
$$
Moreover, if $\Omega\subset\mathbb{R}^N$ is an open set, we define
$$
J_\Omega(E)=L(E\cap \Omega, E^c)+L(E\setminus \Omega,E^c\cap \Omega).
$$
Observe that if $E\subset\Omega$, then
\[
J_\Omega(E)=L(E,E^c)=\frac{1}{2}\, P_s(E).
\]
Using this perimeter-type functional we introduce the notion of {\it nonlocal minimal surfaces} and {\it almost nonlocal minimal surfaces}, in the spirit of \cite{CRS10,CG11}.
\begin{defi} We say that $E$ is a {\it nonlocal minimal surface in $\Omega$} if for any $F$ such that $F\setminus\Omega = E\setminus \Omega$ there holds
$$
J_\Omega(E)\leq J_\Omega(F).
$$
\end{defi}

\begin{defi} 
Let $\delta>0$ and $\omega:(0,\delta)\to \mathbb{R}^+$ be a modulus of continuity. Then we say that $E\subset\mathbb{R}^N$ is {\it $(J_\Omega,\omega,\delta)-$minimal in $\Omega$} if for any  $x_0\in \partial E$ and any set $F$ such that $E\setminus B_r(x_0)=F\setminus B_r(x_0)$ and $r<\min (\delta, \mathrm{dist}(x_0,\partial\Omega))$ we have
$$
J_\Omega(E)\leq J_\Omega (F)+\omega(r)\,r^{N-s}.
$$
We will also simply say that $E$ is {\it almost minimal} in $\Omega$.
\end{defi}
\noindent
We need to recall the following regularity result.

\begin{teo} 
\label{teo:reg}
Assume that $E\subset\mathbb{R}^N$ is $(J_\Omega,C\, r^\alpha,1)-$minimal in $B_1$ for some $\alpha\in (0,s]$ and some $C>0$. Then:

\begin{enumerate}\item  there exists $\delta_0=\delta_0(N,s,\alpha,C)>0$ such that if 
\[
\partial E \cap B_1\subset \{x\, :\, |\langle x,e\rangle|\le \delta_0\},\qquad \mbox{ for some unit vector } e,
\] 
then $\partial E$ is $C^1$ in $B_{1/2}$;
\item outside a singular set having at most Hausdorff dimension $N-2$, $\partial E$ is $C^1$ regular; 
\item in the case $N=2$, the singular set is empty, i.e., $\partial E$ is $C^1$ regular everywhere.
\end{enumerate}
\end{teo}

\begin{proof}
The first two parts are proved in \cite{CG11}. For the last part we observe that in \cite[Theorem 1]{SV12} it is proved that actually there are no singular cones for $N=2$. By using \cite[Theorem 7.4, part 3]{CG11} this implies that nonlocal almost minimal surfaces are $C^1$ for $N=2$ as well.
\end{proof}
By appealing to the previous result, we can prove our first interior regularity result for an $s-$Cheeger set.
\begin{prop}[Interior regularity] 
\label{coro:C1} 
Let $s\in(0,1)$ and $\Omega\subset\mathbb{R}^N$ be an open and bounded set. Let $E$ be an $s-$Cheeger set of $\Omega$. Then $\partial E\cap \Omega$ is $C^1$, up to a singular set of Hausdorff dimension at most $N-2$. 
In the case $N=2$, $\partial E\cap \Omega$ is $C^1$.
\end{prop}
\begin{proof} We prove at first that $E$ is $(J_\Omega,C\,r^s,1)-$minimal in $\Omega$, with $C=C(N,\Omega)>0$. Since $E$ is an $s-$Cheeger set, it is a minimizer of
\begin{equation}
\label{additivo}
2\,\int_{E}\int_{E^c}\frac{1}{|x-y|^{N+s}} \,dx\, dy-h_s(\Omega)|E|,
\end{equation}
among all subsets of $\Omega$.
Hence, for any $x_0\in \partial E\cap \Omega$, $r<\mathrm{dist}(x_0,\partial \Omega)$ and $F$ such that $F\setminus B_r(x_0)=E\setminus B_r(x_0)$, we have $F\subset\Omega$ and thus
\begin{align*}
2\,\int_{E}\int_{E^c}\frac{1}{|x-y|^{N+s}} \,dx\, dy&\leq 2\,\int_{F}\int_{F^c}\frac{1}{|x-y|^{N+s}} \,dx\, dy+h_s(\Omega)(|E\cap B_r(x_0)|-|F\cap B_r(x_0)|)\\
&\leq 2\,\int_{F}\int_{F^c}\frac{1}{|x-y|^{N+s}} \,dx\, dy+C\,r^N.
\end{align*}
Now, we observe that $E\setminus \Omega=\emptyset=F\setminus\Omega$, thus the previous estimate is the same as
\begin{equation}
\label{almostminimal}
J_\Omega(E)\leq J_\Omega(F)+C\,r^N,
\end{equation}
which proves that $E$ is $(J_\Omega,C\,r^s,1)-$minimal in $\Omega$.
\vskip.2cm\noindent
We are now going to use Theorem \ref{teo:reg}. Let $x_0\in \partial E\cap \Omega$, then there exists a ball $B_{r_0}(x_0)\subset\Omega$. By defining
\[
\widetilde E=\frac{E-x_0}{r_0}\qquad \mbox{ and }\qquad\widetilde\Omega =\frac{\Omega-x_0}{r_0},
\] 
and using the scaling properties of $L$, we get
\[
J_{\widetilde\Omega}(\widetilde E)=\frac{J_\Omega(E)}{r_0^{N-s}}.
\] 
By using this and \eqref{almostminimal} we get
\[
J_{\widetilde\Omega}(\widetilde E)\le J_{\widetilde \Omega} (\widetilde F)+\frac{C}{r_0^N}\,r^N,
\]
for every $y\in\partial\widetilde E$, every $\widetilde F$ such that $\widetilde E\setminus B_r(y)=\widetilde F\setminus B_r(y)$ and every $r$ such that $r<\mathrm{dist}(y,\partial\widetilde\Omega)$. This gives that $\widetilde E$ is $(J_{\widetilde\Omega},\widetilde C\, r^s,1)-$minimal in $\widetilde \Omega$, where $\widetilde C=C\, r_0^{-N}$. Observe that $B_1\subset\widetilde\Omega$ by construction, thus $\widetilde E$ has the same almost minimality property in $B_1$ and Theorem \ref{teo:reg} applies. By scaling and translating back, we get the desired result for $E$. 
\end{proof}
By the same idea, we can obtain regularity of an $s-$Cheeger set at points touching $\partial\Omega$.
\begin{prop}[Boundary regularity]
\label{prop:boundary} Let $x_0\in \partial E\cap \partial \Omega$ and assume that $\partial \Omega$ is locally of class $C^{1,\alpha}$ around $x_0$. Then there exists $r_0>0$ such that $\partial E\cap B_{r_0}(x_0)$ is the graph of a $C^1$ function.
\end{prop}
\begin{proof} Let $r_0>0$ and set for simplicity $B=B_{r_0}(x_0)$. Up to translating and scaling the sets as in the proof Proposition \ref{coro:C1}, we can suppose for simplicity that $x_0=0$ and $r_0=1$. As before,  we start by proving that $E$ is $(J_B,C\,r^{\widetilde\alpha},1)-$minimal in $B$, where we set
\[
\widetilde \alpha=\min\{\alpha,s\}.
\]
We take again $F$ to be a set coinciding with $E$ outside $B_r(y)$ for $y\in \partial E\cap B$ and $r<\mathrm{dist}(y,\partial B)$. Then $F\cap \Omega$ is admissible for the minimization of \eqref{additivo}, thus as before
\begin{align*}
\int_{E}\int_{E^c}\frac{1}{|x-y|^{N+s}} \,dx\, dy&\leq \int_{F\cap \Omega}\int_{(F\cap \Omega)^c}\frac{1}{|x-y|^{N+s}} \,dx\, dy+h_s(\Omega)\,(|E|-|F\cap \Omega|)\\&\leq \int_{F\cap \Omega}\int_{(F\cap \Omega)^c}\frac{1}{|x-y|^{N+s}} \,dx\, dy+C\,r^N, 
\end{align*}
where in the second inequality we used that $E$ and $F$ only differ in $B_r(y)$. For the same reason we have $F\setminus B=E\setminus B$ and $F^c\setminus B=E^c\setminus B$, so that
\[
\begin{split}
J_B(E)&=L(E\cap B,E^c)+L(E\setminus B,E^c\cap B)\\
&=L(E\cap B,E^c)+L(E\setminus B,E^c)+L(E\setminus B,E^c\cap B)-L(E\setminus B,E^c)\\
&=L(E,E^c)-L(E\setminus B,E^c\setminus B)\\
&\le L(F\cap\Omega,(F\cap \Omega)^c)-L(F\setminus B,F^c\setminus B)+C\,r^N\\
&=J_B(F)+\Big[L(F\cap\Omega,(F\cap\Omega)^c)-L(F\cap B,F^c)-L(F\setminus B,F^c)\Big]\\
&+C\, r^N\\
\end{split}
\]
which gives
\begin{equation}
\label{erik}
J_B(E)\le J_B(F)+\Big[L(F\cap\Omega,(F\cap\Omega)^c)-L(F,F^c)\Big]+C\, r^N.
\end{equation}
We have to estimate the second term in the right-hand side of \eqref{erik}. For this, we note that
\[
F^c\cup (F\cap \Omega^c)=F^c\cup \Omega^c=(F\cap \Omega)^c,
\]
then for every positive measurable function $g$ we have
\[
\begin{split}
\int_{F\cap \Omega}\int_{(F\cap \Omega)^c}\!\! g(x,y)\, dxdy &=\int_{F\cap\Omega}\int_{F^c} g(x,y)\, dxdy +\int_{F\cap \Omega}\int_{F\cap\Omega^c} g(x,y)\, dxdy \\
&=\int_F\int_{F^c}\!g(x,y)\, dxdy-\int_{F\cap \Omega^c}\int_{F^c}g(x,y)\, dxdy+\int_{F\cap \Omega}\int_{F\cap \Omega^c}\! g(x,y)\, dxdy\\
&\le \int_F\int_{F^c}g(x,y)\, dxdy+\int_{\Omega}\int_{B_r(y)\cap \Omega^c}\! g(x,y)\, dxdy,
\end{split}
\]
thanks to the fact that $F\cap \Omega^c\subset B_r(y)\cap \Omega^c$, since $F\subset B_r(y)\cup\Omega$ by construction. Thus we can infer
\begin{align*}
\int_{F\cap \Omega}\int_{(F\cap \Omega)^c}\frac{1}{|x-y|^{N+s}} \,dx\, dy&\leq \int_{F}\int_{F^c}\frac{1}{|x-y|^{N+s}} \,dx\, dy+\int_{\Omega}\int_{B_r(y)\cap\Omega^c}\frac{1}{|x-y|^{N+s}} \,dx\, dy\\&\leq \int_{F}\int_{F^c}\frac{1}{|x-y|^{N+s}} \,dx\, dy+C\,r^\alpha\, r^{N-s}, 
\end{align*}
where the second inequality follows from \cite[Example 2]{CG11} since we have assumed that $\Omega$ has a $C^{1,\alpha}$ boundary. By inserting the previous estimate in \eqref{erik}, we finally get 
$$
J_B(E)\leq J_B(F)+C\,r^N+C\,r^\alpha\, r^{N-s}, 
$$
which proves that $E$ is $(J_B,C\,r^{\widetilde\alpha},1)-$minimal in $B$, possibly with a different constant $C$.
\vskip.2cm\noindent
The $C^{1}$ regularity now follows from part 3 in \cite[Theorem 7.4]{CG11}. Indeed, if we perform a blow-up of $E$, we will as usual obtain a nonlocal minimal cone $K$. Moreover, the complement $K^c$ is minimal as well and $K^c$ contains a tangential ball, due to the fact that $\Omega$ is assumed to be $C^{1,\alpha}$, which means that $\partial\Omega$ becomes a half-space after a blow-up.  By \cite[Corollary 6.2]{CRS10} we get that $\partial K$ is a $C^1$ surface, and since $K$ is a cone, this means that $K$ is a half-space. From \cite[Theorem 7.4, part 3]{CG11}, we can now conclude that $E$ is $C^1$.
\end{proof}
Finally we prove that at any point of $\partial E\cap \Omega$ having a tangent ball from both sides, an $s-$Cheeger set $E$ has constant nonlocal mean curvature equal to $-h_s(\Omega)$. 
At this aim, we first need a technical result, whose proof closely follows that of \cite[Theorem 5.1]{CRS10}.
\begin{lm} 
\label{lm:oneside}
Let $\Omega\subset\mathbb{R}^N$ be an open bounded set and $E\subset\mathbb{R}^N$ a set satisfying
\begin{equation}
\label{eq:super}
L(A,E)-L(A,(E\cup A)^c)\leq C_0\, |A|,
\end{equation}
for every $A\subset \Omega\setminus E$ and for some constant $C_0$. Let us suppose that there exists a ball $B_r(y_0)\subset E$ which is tangent at $x_0\in\partial E\cap\Omega$. Then we have 
\begin{equation}
\label{oneside}
\limsup_{\delta\to 0^+}\int_{\mathbb{R}^N\setminus B_\delta(x_0)}\frac{1_E(x)-1_{E^c}(x)}{|x-x_0|^{N+s}}\, d x\le C_0.
\end{equation} 
\end{lm}
\begin{proof}
We briefly recall the construction of the proof in \cite[Theorem 5.1]{CRS10} for the reader's convenience. Let us set $e_N=(0,\dots,0,1)$, without loss of generality we can assume that $x_0=0$ and that $B_r(y_0)=B_2(-2\,e_N)$, since we can always reduce to this case by rescaling and translating. Take $0<\delta\ll 1$ such that $B_\delta(0)\subset\Omega$ and $0<\varepsilon\ll\delta$ such that\footnote{This is possible by taking for example $\varepsilon\le \varepsilon_0(\delta)$, where
\[
\varepsilon_0(\delta):=\sqrt{1+\frac{\delta^2}{2}}-1\simeq \frac{\delta^2}{4}.
\]} 
$B_{1+\varepsilon}(-e_N)\setminus E\subset B_\delta(0)$. We denote by $\mathcal{T}$ the radial reflection in the sphere $\partial B_{1+\varepsilon}(-e_N)$ (see Figure 1), then
\begin{figure}
\includegraphics[scale=.3]{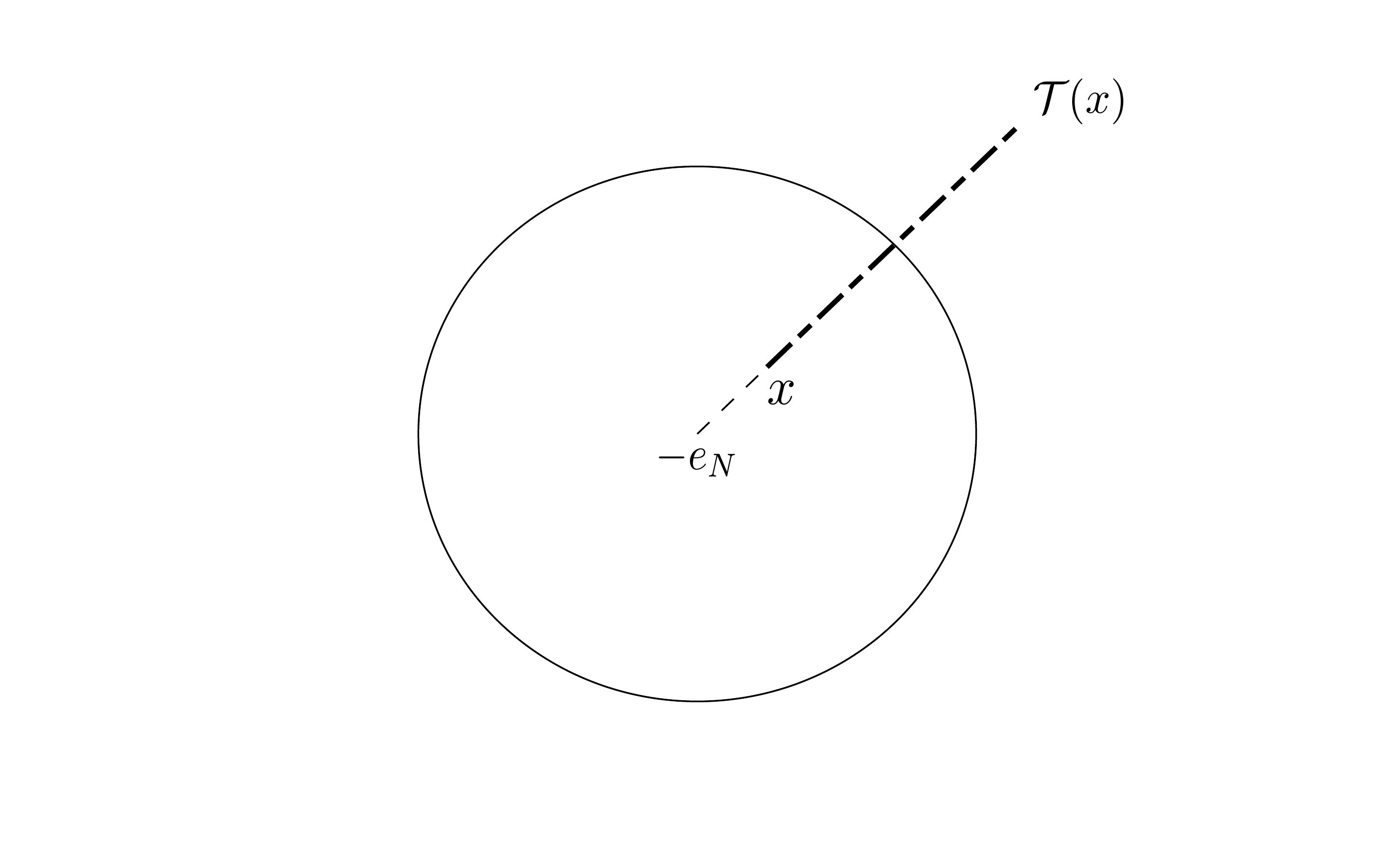}
\caption{The vector $\mathcal{T}(x)+e_N$ is parallel to $x+e_N$ and $\mathcal{T}(x)$ and $x$ have the same distance from the boundary of the ball.}
\end{figure}
we define the sets
$$
A^-=B_{1+\varepsilon}(-e_N)\setminus E,\qquad A^+=\mathcal{T}(A^-)\setminus E\qquad \mbox{ and }\qquad A=A^+\cup A^-,
$$
see Figure 2.
\begin{figure}
\includegraphics[scale=.35]{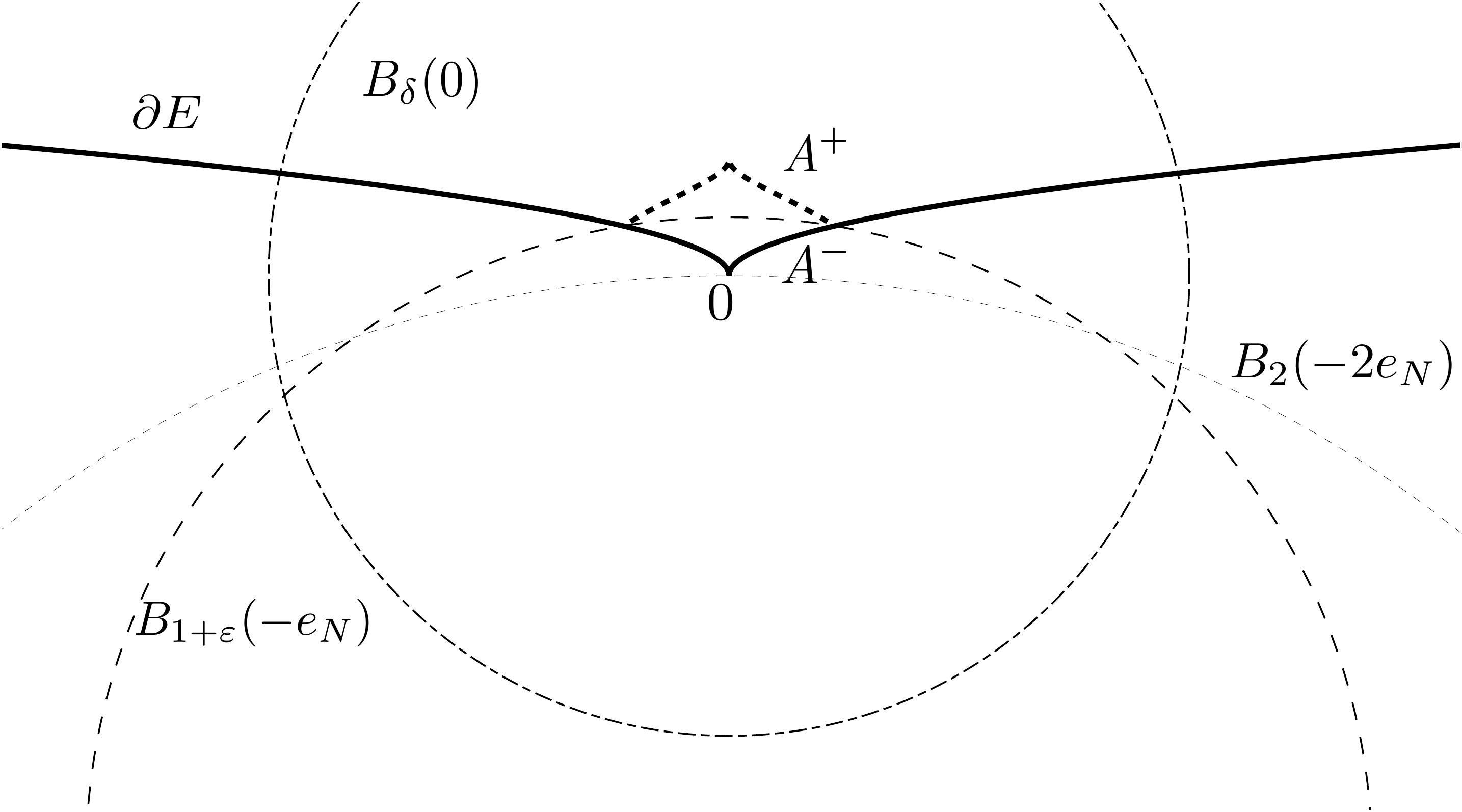}
\caption{The construction of Lemma \ref{lm:oneside}. In this particular case we have $\mathcal{T}(A^-)\cap E=\emptyset$, so that we simply have $A^+=\mathcal{T}(A^-)$.}
\end{figure}
Observe that by construction $A\subset B_\delta(0)\subset\Omega$.
Finally, we define 
$$
F:=\mathcal{T}(B_\delta(0)\cap (E\cup A)^c)\subset E\cap B_\delta(0).
$$
Since by construction $A\subset \Omega\setminus E$, by \eqref{eq:super} we get
\begin{equation}
\label{eq:supsoluse}
L(A,E)-L(A,(E\cup A)^c)\leq C_0\,|A|.
\end{equation}
We remark that we have the following set relations
\begin{align*}
E&=(E\setminus B_\delta(0)) \cup [(E\cap B_\delta(0)) \setminus F]\cup [(E\cap B_\delta(0))\cap F]\\
&=(E\setminus B_\delta(0)) \cup [(E\cap B_\delta(0))\setminus F]\cup F,
\end{align*} 
where we used that $F\subset E\cap B_\delta(0)$.
Also, since $A\subset B_\delta(0)$ 
there holds
\begin{align*}
(E\cup A)^c=E^c\cap A^c &= [(E^c\cap A^c)\cap B_\delta(0)] \cup[(E^c\cap A^c) \setminus B_\delta(0)] \\
&=\mathcal{T}(F) \cup[E^c\setminus B_\delta(0)].
\end{align*}
By putting these two relations together we can realize that
\[
\begin{split}
L&(A,E)-L(A,(E\cup A)^c)\\
&=\Big[L(A,E\setminus B_\delta(0))-L(A,E^c\setminus B_\delta(0))\Big]-\Big[L(A,\mathcal{T}(F))-L(A,F)\Big]+L(A,(E\cap B_\delta(0))\setminus F)\\
&\geq \Big[L(A,E\setminus B_\delta(0))-L(A,E^c\setminus B_\delta(0))\Big]-\Big[L(A,\mathcal{T}(F))-L(A,F)\Big]=:I_1-I_2.
\end{split}
\]
Due to \eqref{eq:supsoluse} we can thus conclude that
\begin{equation}\label{eq:I1I2}
I_1-I_2\leq C_0\,|A|.
\end{equation}
In the proof of \cite[Theorem 5.1]{CRS10} it is proved that (see formula (5.1))
$$
\left|I_1-|A|\,\int_{\mathbb{R}^N\setminus B_\delta(0)}\frac{1_E(y)-1_{E^c}(y)}{|y|^{N+s}}\, dy \right|\leq C\,\varepsilon^\frac12\,\delta^{-1-s}\,|A|,
$$
and\footnote{We should note that in order to prove the estimate on $I_2$ the authors strongly use the {\it positive density property} for nonlocal minimal surfaces (see \cite[Section 4]{CRS10}), which also holds true for nonlocal almost minimal boundaries, see \cite[Proposition 4.1]{CG11}.} (see formula (5.3) and Lemma 5.2 in \cite{CRS10})
$$
I_2\leq C\,\delta^{1-s}\,|A|+C\,\varepsilon^\eta\,|A^-|\le C\, |A|\, (\delta^{1-s}+\varepsilon^\eta),
$$
for some sequence $\varepsilon\to 0$ and for some $\eta\in (0,1-s)$.  Plugging the two above estimates into \eqref{eq:I1I2} and dividing by $|A|$ yields
$$
\int_{\mathbb{R}^N\setminus B_\delta(0)}\frac{1_E(y)-1_{E^c}(y)}{|y|^{n+s}}\, dy\leq C_0+C\,\varepsilon^\frac12\,\delta^{-1-s}+C\,\delta^{1-s}+C\,\varepsilon^\eta.
$$
We then pass to the limit as $\varepsilon$ goes to $0$ and then to the limit as $\delta$ goes to $0$, this implies
$$
\limsup_{\delta\to 0^+}\int_{\mathbb{R}^N\setminus B_\delta(0)}\frac{1_E(y)-1_{E^c}(y)}{|y|^{N+s}}\, dy\leq C_0,
$$
which concludes the proof.
\end{proof}
\begin{teo}
\label{teo:CMC}
Let $E$ be an $s-$Cheeger set of $\Omega$ such that $E$ admits a tangent ball from both sides at $x_0\in \partial E\cap \Omega$. Then 
\[
\lim_{\delta\to 0^+}\int_{\mathbb{R}^N\setminus B_\delta(x_0)} \frac{1_E(x)-1_{E^c}(x)}{|x-x_0|^{N+s}}\, dx=-h_s(\Omega).
\] 
\end{teo}
\begin{proof}
We first observe that since $E$ is a minimizer of \eqref{additivo}, we get that $E$ satisfies \eqref{eq:super} with $C_0=-h_s(\Omega)$. For this, it sufficient to
test the minimality of $E$ against a set of the form $A\cup E$, for every $A\subset \Omega\setminus E$.  Therefore \eqref{oneside} implies
\[
\limsup_{\delta\to 0^+}\int_{\mathbb{R}^N\setminus B_\delta(x_0)}\frac{1_E(x)-1_{E^c}(x)}{|x-x_0|^{N+s}}\, d x\le -h_s(\Omega).
\] 
On the other hand, we also get
\begin{equation}\label{eq:sub}
L(A,E\setminus A)-L(A,E^c)\geq -h_s(\Omega)\, |A|,
\end{equation}
where this time we tested the minimality of $E$ against $E\setminus A$, with $A$ being any subset of $E$. It is immediate to see that \eqref{eq:sub} means that $E^c$ as well satisfies \eqref{eq:super}, this time with $C_0=h_s(\Omega)$ and by hypothesis $E^c$ contains a tangent ball at $x_0$. Then again we can apply Lemma \ref{lm:oneside} and thus
\[
-h_s(\Omega)\le -\limsup_{\delta\to 0^+} \int_{\mathbb{R}^N\setminus B_\delta(x_0)} \frac{1_{E^c}(x)-1_{(E^c)^c}(x)}{|x-x_0|^{N+s}}\, dx=\liminf_{\delta\to 0^+}\int_{\mathbb{R}^N\setminus B_\delta(x_0)} \frac{1_{E}(x)-1_{E^c}(x)}{|x-x_0|^{N+s}}\, dx,
\]
which gives the desired result.
\end{proof}

\section{The first eigenvalues and the Cheeger constant} 
\label{sec:lim}

In this section we show that for a Lipschitz set $\Omega$, the nonlocal Cheeger constant $h_s(\Omega)$ is the limit of the first eigenvalues $\lambda^s_{1,p}(\Omega)$, as in the case of the $p-$Laplacian. The main result is the following.
\begin{teo}[Convergence of the minima]
\label{teo:lim}
Let $\Omega\subset\mathbb{R}^N$ be an open and bounded Lipschitz set. For every $0<s<1$ we have
\begin{equation}
\label{limite}
\lim_{p\to 1^+} \lambda^s_{1,p}(\Omega)=h_s(\Omega).
\end{equation}
\end{teo}
\begin{proof} 
We are going to prove the two inequalities
\[
\limsup_{p\to 1^+} \lambda^s_{1,p}(\Omega)\le h_s(\Omega)\qquad \mbox{ and }\qquad \liminf_{p\to 1^+} \lambda^s_{1,p}(\Omega)\ge h_s(\Omega).
\]
{\it Limsup inequality}.  
We have for any $\varphi\in C_0^\infty(\Omega)$
$$
\lambda^s_{1,p}(\Omega)\leq \left(\frac{[\varphi]_{W^{s,p}(\mathbb{R}^N)}}{\|\varphi\|_{L^p(\Omega)}}\right)^p.
$$
Thus, 
\begin{equation}
\label{limsup}
\limsup_{p\to 1^+} \lambda^s_{1,p}(\Omega)\leq\limsup_{p\to 1^+}\left(\frac{[\varphi]_{W^{s,p}(\mathbb{R}^N)}}{\|\varphi\|_{L^p(\Omega)}}\right)^p=\frac{[\varphi]_{W^{s,1}(\mathbb{R}^N)}}{\|\varphi\|_{L^1(\Omega)}}.
\end{equation}
Thanks to equation \eqref{uguali} and by density of $C^\infty_0(\Omega)$ in $\widetilde W^{s,1}_0(\Omega)$, for every $\delta>0$ we can take $\varphi_\delta\in C_0^\infty(\Omega)$ so that 
$$
h_s(\Omega)+\delta\geq \frac{[\varphi_\delta]_{W^{s,1}(\mathbb{R}^N)}}{\|\varphi_\delta\|_{L^1(\Omega)}}.
$$ 
Then by appealing to \eqref{limsup} we get
$$
\limsup_{p\to 1^+} \lambda^s_{1,p}(\Omega)\leq h_s(\Omega)+\delta.
$$
Since $\delta$ is arbitrary, this proves the limsup inequality.
\vskip.2cm\noindent
{\it Liminf inequality}. Let $\{p_j\}_{j\in\mathbb{N}}\subset(1,+\infty)$ be a sequence converging to $1$ and such that
$$
\lim_{j\to\infty}\lambda^s_{1,p_j}(\Omega)=\liminf_{p\to 1^+} \lambda^s_{1,p}(\Omega),
$$
and let $u_{p_j}\in \widetilde W^{s,p_j}_0(\Omega)$ achieve the minimum in \eqref{p}.
Thanks to Lemma \ref{lm:sommabilità} we have the continuous embedding 
$$
\widetilde W^{s,p_j}_0(\Omega)\hookrightarrow \mathcal{W}^{s/2,1}_0(\Omega).
$$ 
More precisely, for $j$ large enough we can infer
\[
[u_{p_j}]_{W^{s/2,1}(\mathbb{R}^N)}\le C\, [u_{p_j}]_{W^{s,p_j}(\mathbb{R}^N)}=C\, \lambda^s_{1,p_j}(\Omega)^\frac{1}{p_j}\le C\, (1+h_s(\Omega)),
\]
for a constant $C>0$ which does not depends on $p_j$. By Corollary \ref{coro:compact}, up to extracting a subsequence, we can then suppose that the eigenfunctions $\{u_{p_j}\}_{j\in\mathbb{N}}$ are converging to a function $u$ in $L^q(\Omega)$ and almost everywhere, for an exponent $1<q<2N/(2N-s)$. In particular we have
\[
\lim_{j\to\infty}\Big|\|u_{p_j}\|_{L^{p_j}}-\|u\|_{L^1}\Big|\le \lim_{j\to\infty}\|u_{p_j}-u\|_{L^{p_j}}+\lim_{j\to\infty}\Big|\|u\|_{L^{p_j}}-\|u\|_{L^1}\Big|=0,
\]
where we used that $p_j\le q$ for $j$ sufficiently large. The previous implies that $\|u\|_{L^1}=1$.
From Fatou's Lemma we can then infer
$$
\liminf_{p\to 1^+} \lambda^s_{1,p}(\Omega)=\lim_{j\to\infty}\lambda^s_{1,p_j}(\Omega)=\lim_{j\to \infty}[u_{p_j}]^{p_j}_{W^{s,p_j}(\mathbb{R}^N)}\geq [u]_{W^{s,1}(\mathbb{R}^N)}\ge h_s(\Omega),
$$
where we used \eqref{pensaci!} in the last inequality.
\end{proof}
\begin{teo}[Convergence of the minimizers]
Let $\Omega\subset\mathbb{R}^N$ be an open and bounded Lipschitz set and for every $p>1$ let $u_p$ achieve the minimum \eqref{p}. Then there exists a sequence $\{p_j\}_{j\in\mathbb{N}}$ converging to $1$ such that $\{u_{p_j}\}_{j\in\mathbb{N}}$ converges strongly in $L^q(\Omega)$ for every $q<\infty$  to a solution $u_1$ of 
\[
\lambda^s_{1,1}(\Omega)=\min_{u\in \mathcal{W}^{s,1}_0(\Omega)} \left\{[u]_{W^{s,1}(\mathbb{R}^N)}\, :\, \|u\|_{L^1(\Omega)}=1,\ u\ge 0\right\}=h_s(\Omega).
\]
Moreover $u_1\in L^\infty(\Omega)$ and we have
\begin{equation}
\label{cag}
\|u_1\|_{L^\infty(\Omega)} \le \left[\frac{|B|^\frac{N-s}{N}}{P_s(B)}\right]^{\frac{N}{s}}\,h_s(\Omega)^\frac{N}{s},
\end{equation}
where $B$ is any $N-$dimensional ball. 
\end{teo}
\begin{proof}
Observe that we have 
\[
[u_p]^p_{W^{s,p}(\mathbb{R}^N)}=\lambda^s_{1,p}(\Omega),
\]
then by Theorem \ref{teo:lim}
\[
\lim_{p\to 1^+} [u_p]^p_{W^{s,p}(\mathbb{R}^N)}=h_s(\Omega).
\]
As in the proof of Theorem \ref{teo:lim}, by Lemma \ref{lm:sommabilità} $\{u_p\}_{p>1}$ is equi-bounded in $\mathcal{W}^{s/2,1}_0(\Omega)$ for $p$ sufficiently small. Again thanks to Theorem \ref{teo:compact} we can extract a subsequence $\{u_{p_j}\}_{j\in\mathbb{N}}$ converging in $L^1$ to a function $u_1$ such that $\|u_1\|_{L^1}=1$. Thus we obtain
\[
h_s(\Omega)=\lim_{j\to \infty} [u_{p_j}]^{p_j}_{W^{s,p_j}(\mathbb{R}^N)}\ge [u_1]_{W^{s,1}(\mathbb{R}^N)}\ge h_s(\Omega),
\]
thus $u_1$ achieves $\lambda^s_{1,1}(\Omega)=h_s(\Omega)$.
Observe that the sequence $\{u_{p_j}\}_{j\in\mathbb{N}}$ is equi-bounded in $L^\infty(\Omega)$ thanks to Proposition \ref{lm:stimetta}, then $u_1$ as well is in $L^\infty(\Omega)$ since
\begin{equation}
\label{pax}
\|u_1\|_{L^\infty}\le \liminf_{j\to\infty} \|u_{p_j}\|_{L^\infty}\le \lim_{j\to\infty} \widetilde C_{N,p_j,s}\, \lambda^s_{1,p_j}(\Omega)^\frac{N}{s\,p_j^2}<+\infty.
\end{equation}
By a simple interpolation argument we then get that $\{u_{p_j}\}_{j\in\mathbb{N}}$ actually converges to $u_1$ in every $L^q(\Omega)$, with $1\le q<\infty$.
\vskip.2cm\noindent
In order to prove \eqref{cag}, we use \eqref{pax} and keep into account Remark \ref{oss:costante}, which permits to infer
\[
\limsup_{j\to\infty} \widetilde C_{N,p_j,s}=\limsup_{j\to\infty} \left(\sup_{u\in W^{s,p_j}_0(\mathbb{R}^N)\setminus\{0\}} \frac{\displaystyle\left(\int_{\mathbb{R}^N} |u|^\frac{N\,p_j}{N-s\,p_j}\, dx\right)^\frac{N-s\,p_j}{N}}{[u]^p_{W^{s,p_j}(\mathbb{R}^N)}}\right)^\frac{N}{s\,p_j^2}\le \left[\frac{|B|^\frac{N-s}{N}}{P_s(B)}\right]^{\frac{N}{s}}.
\]
In the last inequality we have used that the limsup of the best constant of the Sobolev inequality in $W^{s,p}_0(\mathbb{R}^N)$ for $p>1$ is certainly less than the best constant for the limit case\footnote{This is a consequence of \cite[Corollary 4.2]{FS} with $r=p^*$ and equation (4.2) there.} $p=1$, the latter being given by \eqref{isoper}.
Combining this and the convergence of $\lambda^s_{1,p_j}(\Omega)$ to $h_s(\Omega)$ concludes the proof.
\end{proof}
\begin{oss}
Actually, estimate \eqref{cag} holds true for {\it every} function $u\in\mathcal{W}^{s,1}_0(\Omega)$ attaining $\lambda^s_{1,1}(\Omega)$. Indeed, let $\Omega_t=\{x\in\Omega\, :\, u(x)>t\}$ and set $M=\mathrm{ess}\sup\{t\ge 0\,: \, |\Omega_t|>0\}$. By Theorem \ref{teo:primo1}, we know that $\Omega_t$ is a $s-$Cheeger set of $\Omega$, for almost every $t\in[0,M)$. We then get
\[
\begin{split}
1= \left(\frac{|\Omega_t|}{|\Omega_t|}\right)^\frac{N}{s}&=|\Omega_t|^\frac{N}{s}\,\left(\frac{h_s(\Omega)}{P_s(\Omega_t)}\right)^\frac{N}{s}\le |\Omega_t|^\frac{N}{s}\,\left(\frac{h_s(\Omega)}{P_s(B)}\right)^\frac{N}{s}\, \left(\frac{|B|}{|\Omega_t|}\right)^\frac{N-s}{s},
\end{split}
\]
where we used the isoperimetric inequality \eqref{siso}. Thus the previous gives
\[
1\le |\Omega_t|\, \left[\frac{|B|^\frac{N-s}{N}}{P_s(B)}\right]^\frac{N}{s}\, h_s(\Omega)^\frac{N}{s},\qquad \mbox{ for a.e. }t\in[0,M).
\] 
By integrating the previous in $t\in[0,M)$ and using Cavalieri principle, we get \eqref{cag} for $u$.
\par 
Observe that equality holds in \eqref{cag} when $\Omega=B$ is a ball and $u=1_B$. Thus the $L^\infty$ estimate \eqref{stimettaLinfty} becomes sharp in the limit as $p$ goes to $1$. In the local case, we recall that an $L^\infty$ estimate for functions attaining the Cheeger constant $h_1(\Omega)$
can be found in \cite[Theorem 4]{CM}.
\end{oss}
We have already seen in Proposition \ref{prop:hsh1} that $h_s(\Omega)$ can be estimated in terms of $h_1(\Omega)$. By using the recent $\Gamma-$convergence result by Ambrosio, De Philippis and Martinazzi in \cite{ADePM}, one can show that $h_s(\Omega)$ converges to $h_1(\Omega)$, as $s$ goes to $1$.
\begin{prop}
Let $\Omega\subset \mathbb{R}^N$ be an open and bounded set. Then we have
\begin{equation}
\label{scostanti}
\lim_{s\nearrow 1} (1-s)\, h_s(\Omega)=2\,\omega_{N-1}\, h_1(\Omega).
\end{equation}
\end{prop}
\begin{proof}
For every $\varepsilon>0$, let $u_\varepsilon\in C^\infty_0(\Omega)$ be such that
\[
h_1(\Omega)+\varepsilon\ge \frac{|\nabla u_\varepsilon|(\mathbb{R}^N)}{\|u_\varepsilon\|_{L^1(\Omega)}}.
\]
By using \eqref{conv} we obtain
\[
2\,\omega_{N-1}\, (h_1(\Omega)+\varepsilon)\ge 2\,\omega_{N-1}\, \frac{|\nabla u_\varepsilon|(\mathbb{R}^N)}{\|u_\varepsilon\|_{L^1(\Omega)}}=\lim_{s\nearrow 1} (1-s)\, \frac{[u_\varepsilon]_{W^{s,1}(\mathbb{R}^N)}}{\|u_\varepsilon\|_{L^1(\Omega)}}\ge \lim_{s\nearrow 1} (1-s)\, h_s(\Omega),
\]
where we used again \eqref{pensaci!} to get the last estimate.
By the arbitrariness of $\varepsilon$, we get
\[
2\,\omega_{N-1}\, h_1(\Omega)\ge \limsup_{s\nearrow 1} (1-s)\, h_s(\Omega).
\]
On the other hand, let $\{s_j\}_{j\in\mathbb{N}}\subset(0,1)$ be a sequence increasingly converging to $1$ such that
\[
\lim_{j\to\infty} (1-s_j)\, h_{s_j}(\Omega)=\liminf_{s\nearrow 1}(1-s)\, h_s(\Omega).
\]
For every $j\in\mathbb{N}$ let us take $E_j\subset\Omega$ such that
\begin{equation}
\label{equal}
(1-s_j)\, h_{s_j}(\Omega)=(1-s_j)\, \frac{P_{s_j}(E_j)}{|E_j|},
\end{equation}
so that
\[
(1-s_j)\, P_{s_j}(E_j)\le C.
\]
Up to a subsequence (not relabeled), the sequence $\{E_j\}_{j\in\mathbb{N}}$ is then converging in $L^1$ to a Borel set $E_\infty\subset\Omega$, thanks to \cite[Theorem 1]{ADePM}. This implies in particular that $|E_{j}|$ converges to $|E_\infty|$, but we have to exclude that $|E_\infty|=0$. At this aim we observe that by \eqref{equal}, proceeding as in the proof of Proposition 5.3, we get
\[
|E_j|\ge \left(\frac{1}{C_{N,\Omega}} (1-s_j)\, P_{s_j}(B)\, |B|^\frac{s_j-N}{N}\right)^\frac{N}{s_j},
\]
where $B$ is any $N-$dimensional ball. By passing to the limit as $j$ goes to $\infty$ in the previous and using \eqref{conv}
\[
|E_\infty|\ge \left(\frac{2\,\omega_{N-1}}{C_{N,\Omega}}\, P(B)\, |B|^\frac{1-N}{N}\right)^N>0,
\] 
as desired. We can now use the $\Gamma-$liminf inequality of \cite[Theorem 2]{ADePM} to infer\footnote{Again, our definition of $P_s(\Omega)$ differs by a multiplicative factor $2$ from that in \cite{ADePM}.}
\[
\liminf_{s\to 1} (1-s)\, h_s(\Omega)\ge \liminf_{j\to \infty} (1-s_j)\, \frac{P_{s_j}(E_{s_j})}{|E_{s_j}|}\ge 2\,\omega_{N-1}\, \frac{P(E_\infty)}{|E_\infty|}\ge 2\, \omega_{N-1}\, h_1(\Omega).
\]
This concludes the proof.
\end{proof}

\section{A nonlocal Max Flow Min Cut Theorem}
\label{sec:maxmin}

It is well-known (see \cite{Gr}) that for the Cheeger constant $h_1$, we have the following dual characterization in terms of vector fields with prescribed divergence
\[
\frac{1}{h_1(\Omega)}=\min_{V\in L^\infty(\Omega;\mathbb{R}^N)}\left\{\|V\|_{L^\infty(\Omega)}\, : \, -\mathrm{div\,} V=1\right\},
\]
where the divergence constraint has to be attained in distributional sense, i.e.
\[
\int_\Omega \langle \nabla \varphi,V\rangle\, dx=\int_\Omega \varphi\, dx,\qquad \mbox{ for every } \varphi\in C^\infty_0(\Omega).
\]
The previous in turn can be rewritten as
\[
h_1(\Omega)=\sup\left\{ h\in \mathbb{R}\, :\, \exists V\in L^\infty(\Omega;\mathbb{R}^N) \mbox{ such that } \|V\|_{L^\infty}\le 1 \mbox{ and } -\mathrm{div\,}V\ge h\right\},
\]
and the latter is usually referred to as a continuous version of the Min Cut Max Flow Theorem (see \cite{Gr} for a detailed discussion). In this section we show that similar characterizations hold for $h_s(\Omega)$ as well. 
\vskip.2cm
Let $\Omega\subset\mathbb{R}^N$ be as always an open and bounded set. Let $p\in[1,\infty)$ and $s\in(0,1)$, we set $q=p/(p-1)$ if $p>1$ or $q=\infty$ is $p=1$ and 
\[
\widetilde W^{-s,q}(\Omega)=\left\{F:\widetilde W^{s,p}_0(\Omega)\to\mathbb{R}\, :\, F \mbox{ linear and continuous}\right\}.
\]
We also define the linear and continuous operator $R_{s,p}:\widetilde W^{s,p}_0(\Omega)\to L^p(\mathbb{R}^N\times \mathbb{R}^N)$ by
\[
R_{s,p}(u)(x,y)=\frac{u(x)-u(y)}{|x-y|^{\frac{N}{p}+s}},\qquad \mbox{ for every } u\in \widetilde W^{s,p}_0(\Omega).
\]
\begin{lm}
\label{lm:aggiunto}
The operator $R^*_{s,p}:L^q(\mathbb{R}^N\times\mathbb{R}^N)\to \widetilde W^{-s,q}(\Omega)$ defined by
\begin{equation}
\label{Tstar}
\langle R^*_{s,p}\,(\varphi),u\rangle:=\int_{\mathbb{R}^N}\int_{\mathbb{R}^N} \varphi(x,y)\,\frac{u(x)-u(y)}{|x-y|^{\frac{N}{p}+s}}\, dx\,dy,\qquad \mbox{ for every }u \in \widetilde W^{s,p}_0(\Omega),
\end{equation}
is linear and continuous. Moreover, $R^{*}_{s,p}$ is the adjoint of $R_{s,p}$.
\end{lm}
\begin{proof}
We start by observing that for every $\varphi\in L^q(\mathbb{R}^N\times\mathbb{R}^N)$, $R^*_{s,p}(\varphi)$ defines a distribution on $\Omega$, i.e. $R^*_{s,p}(\varphi)\in\mathcal{D}'(\Omega)$.
Then by H\"older inequality, we get
\begin{equation}
\label{holder}
|\langle R^*_{s,p}(\varphi),u\rangle|\le \|\varphi\|_{L^q(\mathbb{R}^N\times\mathbb{R}^N)}\, \|u\|_{\widetilde W^{s,p}_0(\Omega)}.
\end{equation}
By density this implies that $R^*_{s,p}(\varphi)$ can be (uniquely) extended to an element of $\widetilde W^{-s,q}(\Omega)$ and
\[
\|R^*_{s,p}(\varphi)\|_{\widetilde W^{-s,q}(\Omega)}\le \|\varphi\|_{L^q(\mathbb{R}^N\times\mathbb{R}^N)}.
\] 
Then $R^*_{s,p}$ is well-defined and is of course a linear operator. The previous estimate implies that this is continuous as well.
\vskip.2cm\noindent
To prove the second statement, by the very definition of $R^*_{s,p}$ we get\footnote{Given a topological vector space $X$ and its dual space $X^*$, we denote by $\langle \cdot,\cdot,\rangle_{(X^*,X)}$ the relevant duality pairing.}
\[
\langle R^*_{s,p}(\varphi),u\rangle_{(\widetilde W^{-s,q}(\Omega),\widetilde W^{s,p}_0(\Omega))}=\langle \varphi,R_{s,p}(u)\rangle_{(L^q(\mathbb{R}^N\times \mathbb{R}^N),L^p(\mathbb{R}^N\times \mathbb{R}^N))}.
\]
This concludes the proof.
\end{proof}
\begin{oss}
\label{oss:divergence}
The operator $R^*_{s,p}$ has to be thought of as a sort of nonlocal divergence. Observe that by performing a discrete ``integration by parts'', $R^*_{s,p}(\varphi)$ can be formally written as
\begin{equation}
\label{erascazzata}
R^*_{s,p}(\varphi)(x)=\int_{\mathbb{R}^N} \frac{\varphi(x,y)-\varphi(y,x)}{|x-y|^{\frac{N}{p}+s}}\, dy,\qquad x\in\mathbb{R}^N,
\end{equation}
so that
\[
\langle R^*_{s,p}(\varphi),u\rangle=\int_{\Omega} \left(\int_{\mathbb{R}^N} \frac{\varphi(x,y)-\varphi(y,x)}{|x-y|^{\frac{N}{p}+s}}\, dy\right)\, u(x)\, dx,\qquad u\in \widetilde W^{s,p}_0(\Omega).
\]
Indeed, by using this formula
\[
\begin{split}
\int_{\mathbb{R}^N} u(x)\, R^*_{s,p}(\varphi)(x)\, dx
&=\int_{\mathbb{R}^N}\int_{\mathbb{R}^N} u(x) \frac{\varphi(x,y)}{|x-y|^{\frac{N}{p}+s}}\, dy\, dx-\int_{\mathbb{R}^N}\int_{\mathbb{R}^N} u(x) \frac{\varphi(y,x)}{|x-y|^{\frac{N}{p}+s}}\, dy\, dx,
\end{split}
\]
and exchanging the role of $x$ and $y$ in the second integral, we obtain that this is formally equivalent to \eqref{Tstar}.
\end{oss}
We record the following result for completeness.
\begin{prop}
\label{lm:convexduality}
Let $1<p<\infty$ and $s\in(0,1)$. Given an open and bounded set $\Omega\subset\mathbb{R}^N$, for every $f\in\widetilde  W^{-s,q}(\Omega)$ we have
\begin{equation}
\label{minmax}
\begin{split}
\max_{u\in\widetilde  W^{s,p}_0(\Omega)} &\left\{\langle f,u\rangle-\frac{1}{p}\,\int_{\mathbb{R}^N}\int_{\mathbb{R}^N} \frac{|u(x)-u(y)|^p}{|x-y|^{N+s\,p}}\, dx\,dy
\right\}\\
&=\min_{\varphi\in L^q(\mathbb{R}^N\times\mathbb{R}^N)}\left\{\frac{1}{q} \int_{\mathbb{R}^N}\int_{\mathbb{R}^N} |\varphi|^q\, dx\,dy\, :\, R^*_{s,p}(\varphi)=f \mbox{ in }\Omega\right\},
\end{split}
\end{equation}
where as before $q=p/(p-1)$ and the constraint $R^*_{s,p}(\varphi)=f$ has to be attained in the sense
\[
\langle f,u\rangle=\int_{\mathbb{R}^N}\int_{\mathbb{R}^N} \varphi(x,y)\,\frac{u(x)-u(y)}{|x-y|^{\frac{N}{p}+s}}\, dx\,dy,\qquad \mbox{ for every }u \in C^\infty_0(\Omega).
\]
\end{prop}
\begin{proof}
Observe that the maximization problem in the left-hand side of \eqref{minmax} can be written in the form
\[
\max_{x\in X} \langle x^*,x\rangle-\mathcal{G}(A(x)),\qquad x^*\in X^*,
\]
with $X$ reflexive Banach space having dual $X^*$, $\mathcal{G}:Y\to\mathbb{R}$ a lower semicontinuous convex functional and $A:X\to Y$ a linear continuous operator. Specifically, we have
\[
X=\widetilde W^{s,p}_0(\Omega),\qquad X^*=\widetilde W^{-s,q}(\Omega),\qquad A=R_{s,p},\qquad Y=L^p(\mathbb{R}^N\times\mathbb{R}^N), 
\]
and
\[
\mathcal{G}(\xi)=\frac{1}{p} \|\xi\|^p_{L^p(\mathbb{R}^N\times\mathbb{R}^N)},\qquad \xi\in L^p(\mathbb{R}^N\times\mathbb{R}^N).
\]
Then general duality results of Convex Analysis (see \cite[Proposition 5, page 89]{Ek}) guarantees that
\[
\max_{x\in X} \langle x^*,x\rangle-\mathcal{G}(A(x))=\min\{\mathcal{G}^*(\xi^*)\, :\, A^*(\xi^*)=x^*\},
\]
where $A^*:Y^*\to X^*$ is the adjoint operator of $A$. In our case, we have $Y^*=L^q(\mathbb{R}^N\times\mathbb{R}^N)$ and of course $A^*$ coincides with the operator defined by \eqref{Tstar}, thanks to Lemma \ref{lm:aggiunto}.
\end{proof}
By a simple homogeneity argument, the concave maximization problem in \eqref{minmax} is equivalent to
\[
\max_{u\in \widetilde W^{s,p}_0(\Omega)\setminus\{0\}}\frac{|\langle f,u\rangle|^p}{\displaystyle\int_{\mathbb{R}^N}\int_{\mathbb{R}^N} \frac{|u(x)-u(y)|^p}{|x-y|^{N+s\,p}}\, dx\,dy},
\]
more precisely we have
\[
\begin{split}
\max_{u\in \widetilde W^{s,p}_0(\Omega)}& \left\{\langle f,u\rangle-\frac{1}{p}\,\int_{\mathbb{R}^N}\int_{\mathbb{R}^N} \frac{|u(x)-u(y)|^p}{|x-y|^{N+s\,p}}\, dx\,dy
\right\}\\
&=\left(\frac{p-1}{p}\right)\,\left(\max_{u\in \widetilde W^{s,p}_0(\Omega)\setminus\{0\}}\frac{|\langle f,u\rangle|^p}{\displaystyle\int_{\mathbb{R}^N}\int_{\mathbb{R}^N} \frac{|u(x)-u(y)|^p}{|x-y|^{N+s\,p}}\, dx\,dy}\right)^\frac{1}{p-1}\\
&=\frac{1}{q}\, \|f\|^q_{\widetilde W^{-s,q}(\Omega)},
\end{split}
\]
by recalling that $q=p/(p-1)$. As a straightforward consequence, we have the following.
\begin{coro}
Let $1<q<\infty$ and $s\in(0,1)$, then for every $f\in \widetilde W^{-s,q}(\Omega)$
\[
\|f\|_{\widetilde W^{-s,q}(\Omega)}=\min_{\varphi\in L^q(\mathbb{R}^N\times\mathbb{R}^N)}\left\{\|\varphi\|_{L^q(\mathbb{R}^N\times\mathbb{R}^N)}\, :\, R^*_{s,p}(\varphi)=f \mbox{ in }\Omega\right\}.
\]
\end{coro}
\begin{oss}
By looking at the formal expression \eqref{erascazzata} for $R^*_{s,p}$, we may notice that for a symmetric function $\varphi\in L^q(\mathbb{R}^N\times\mathbb{R}^N)$, i.e. if we have
\[
\varphi(x,y)=\varphi(y,x),
\]
then of course $R^*_{s,p}(\varphi)\equiv 0$. 
Roughly speaking, this means that functions symmetric in the two variables play in this context the same role as free divergence vector fields in the usual local case.
\end{oss}
Then the main result of this section is the following alternative characterization of the $s-$Cheeger constant of a set, which can be used to deduce lower bounds on $h_s(\Omega)$.
\begin{teo}
\label{teo:dualcheeger}
Let $\Omega\subset\mathbb{R}^N$ be an open and bounded Lipschitz set and $s\in(0,1)$. Then we have
\[
\frac{1}{h_s(\Omega)}=\min\big\{\|\varphi\|_{L^\infty(\mathbb{R}^N\times\mathbb{R}^N)}\, :\, R^*_{s,1}(\varphi)=1\big\}.
\]
\end{teo}
\begin{proof}
We start by observing that
\[
\begin{split}
\frac{1}{h_s(\Omega)}&=\sup_{u\in \widetilde W^{s,1}_0(\Omega)\setminus\{0\}} \frac{\displaystyle\int_\Omega |u|\, dx}{\displaystyle\int_{\mathbb{R}^N}\int_{\mathbb{R}^N} \frac{|u(x)-u(y)|}{|x-y|^{N+s}}\,dx\, dy}\\
&=\sup_{u\in \widetilde W^{s,1}_0(\Omega)} \left\{\langle 1,u\rangle\, :\, \int_{\mathbb{R}^N}\int_{\mathbb{R}^N} \frac{|u(x)-u(y)|}{|x-y|^{N+s}}\,dx\, dy\le 1\right\},
\end{split}
\] 
thanks to Lemma \ref{lm:salvaculo} and Theorem \ref{teo:primo1}.
Again, the latter is a problem of the form
\[
\sup_{x\in X} \langle x^*,x\rangle-\mathcal{G}(A(x)),\qquad x^*\in X^*,
\]
where $\mathcal{G}:Y\to\mathbb{R}$ is convex lower semicontinuous and $A:X\to Y$ is linear and continuous. In this case we have
\[
X=\widetilde W^{s,1}_0(\Omega),\qquad X^*=\widetilde W^{-s,\infty}(\Omega),\qquad A=R_{s,1},\qquad Y=L^1(\mathbb{R}^N\times\mathbb{R}^N), 
\]
and
\[
\mathcal{G}(\xi)=\left\{\begin{array}{cc}
0,&\mbox{ if }\displaystyle\|\xi\|_{L^1(\mathbb{R}^N\times\mathbb{R}^N)}\le 1,\\
+\infty, & \mbox{ otherwise}.
\end{array}
\right.
\]
Then again by \cite[Proposition 5]{Ek} we have
\[
\sup_{x\in X} \langle x^*,x\rangle-\mathcal{G}(A(x))=\min_{\xi^*\in Y^*}\left\{\mathcal{G}^*(\xi^*)\, :\, A^*(\xi^*)=x^*\right \},
\]
where 
\[
x^*=1_\Omega,\qquad\mathcal{G}(\xi^*)=\|\xi^*\|_{L^\infty(\mathbb{R}^N\times\mathbb{R}^N)},\qquad Y^*=L^\infty(\mathbb{R}^N\times\mathbb{R}^N),
\]
and the adjoint operator is \(T^*=R^*_{s,1}\).
This concludes the proof.
\end{proof}
As a corollary of the previous result, we obtain the following characterization.
\begin{coro}
Let $\Omega\subset\mathbb{R}^N$ be an open and bounded Lipschitz set and $s\in(0,1)$. Then we have
\begin{equation}
\label{mincut}
h_s(\Omega)=\max\left\{h\in\mathbb{R}\, :\, \exists\, \varphi\in L^\infty(\mathbb{R}^N\times\mathbb{R}^N) \mbox{ s.\,t. } \begin{array}{r}\|\varphi\|_\infty\le 1\mbox{ and }\\ R^*_{s,1}(\varphi)\ge h \mbox{ in }\Omega
\end{array} \right\}.
\end{equation}
\end{coro}
\begin{proof}
We have
\[
\begin{split}
\max\{h\, &:\, \exists\, \varphi\in L^\infty(\mathbb{R}^N\times\mathbb{R}^N) \mbox{ s.\,t. } \|\varphi\|_\infty\le 1 \mbox{ and } R^*_{s,1}(\varphi)\ge h \}\\
&=\max\left\{h\, :\, \exists\, \varphi\in L^\infty(\mathbb{R}^N\times\mathbb{R}^N) \mbox{ s.\,t. } \frac{1}{\|\varphi\|_\infty}\ge h \mbox{ and } R^*_{s,1}(\varphi)\ge 1\right\}\\
&=\max_{\varphi\in L^\infty(\mathbb{R}^N\times\mathbb{R}^N)}\left\{\frac{1}{\|\varphi\|_\infty}\, :\, R^*_{s,1}(\varphi)=1 \mbox{ in }\Omega\right\},
\end{split}
\]
and the latter quantity coincides with
\[
\frac{1}{\min\big\{\|\varphi\|_{L^\infty(\mathbb{R}^N\times\mathbb{R}^N)}\, :\, R^*_{s,1}(\varphi)=1 \mbox{ in }\Omega\big\}}.
\]
By using Theorem \ref{teo:dualcheeger} we can conclude.
\end{proof}

\begin{oss}[Interpretation]
The characterization \eqref{mincut} can be seen as a kind of nonlocal version of the Max Flow Min Cut Theorem. 
A possible interpretation of \eqref{mincut} is the following: we have a continuous network represented by $\mathbb{R}^N$, with sources (producing a given commodity) uniformly distributed in $\Omega$ and the complement of $\Omega$ being the sink. Transportation activities are described by $\varphi$, in such a way that at each point $x\in\Omega$ we have an incoming quantity of flow $\varphi(x,y)\, |x-y|^{-N-s}$ from $y\in\mathbb{R}^N$ and an outcoming flow $\varphi(y,x)\, |x-y|^{-N-s}$ to the same $y\in\mathbb{R}^N$. 
Then the total flow at $x$ is given by (see Remark \ref{oss:divergence})
\[
R^*_{s,1}(\varphi)(x)=\int_{\mathbb{R}^N} \frac{\varphi(x,y)-\varphi(y,x)}{|x-y|^{N+s}}\,dy.
\]
The sources in $\Omega$ continuously in time produce at a rate which is (at least) $h$, that is $R^*_{s,1}(\varphi)\ge h$. The $L^\infty$ bound on $\varphi$ is clearly related to a capacity constraint for our network. A {\it cut} is any $E\subset\Omega$ and observe that for every admissible flow $\varphi$ and every cut $E\subset\Omega$, we (formally) have
\[
h\, |E|\le \int_E R^*_{s,1}(\varphi)\, dx=\int_{\mathbb{R}^N} \int_{\mathbb{R}^N} \varphi\, R_{s,1}(1_E)\, dx\,dy \le P_s(E).
\]
Thus \eqref{mincut} states that trying to find the maximal (nonlocal) flow is the same as trying to find the best (nonlocal) cut of $\Omega$.
\end{oss}

\appendix

\section{Gagliardo seminorms and differential quotients}
\label{app:embedding}
For the sake of completeness, we record the proof of a technical result we needed for the compact embedding $\widetilde W^{s,p}_0(\Omega)\hookrightarrow L^q(\Omega)$. The proof below is an adaptation of that of \cite[Proposition 4]{ADePM}. 
\begin{lm}
\label{lm:nikolski}
Let $1\le p<\infty$ and $0<s<1$, for every $u\in W^{s,p}_0(\mathbb{R}^N)$ there holds
\begin{equation}
\label{nikolski}
\sup_{|h|>0}\int_{\mathbb{R}^N} \frac{|u(x+h)-u(x)|^p}{|h|^{s\,p}}\, dx\le C\,(1-s)\,[u]^p_{W^{s,p}(\mathbb{R}^N)},
\end{equation}
for a constant $C=C(N,p)>0$.
\begin{proof}
Let $\rho\in C^\infty_0(\mathbb{R}^N)$ be a positive function with support given by the annular region $B_1(0)\setminus B_{1/2}(0)=\{x\in\mathbb{R}^N\, :\, 1/2<|x|<1\}$ and such that $\int_{\mathbb{R}^N} \rho\, dx=1$. We fix $h\in\mathbb{R}^N\setminus\{0\}$, then for every $0<\varepsilon<|h|$ we set
\[
\rho_\varepsilon(x)=\frac{1}{\varepsilon^N}\, \rho\left(\frac{x}{\varepsilon}\right),
\]
and we may write
\begin{equation}
\label{mostro}
\begin{split}
|u(x+h)-u(x)|&=\left|\int u(y)\, \rho_\varepsilon(x+h-y)\, dy+\int [u(x+h)-u(x+h-y)]\, \rho_\varepsilon(y)\, dy\right.\\
&\left.-\left(\int u(y)\, \rho_\varepsilon(x-y)\, dy+\int [u(x)-u(x-y)]\, \rho_\varepsilon(y)\, dy\right)\right|\\
&\le \left|\int u(y)\, [\rho_\varepsilon(x+h-y)-\rho_\varepsilon(x-y)]\, dy\right|\\
&+\int |u(x+h)-u(x+h-y)|\, \rho_\varepsilon(y)\, dy+\int |u(x)-u(x-y)|\, \rho_\varepsilon(y)\, dy.
\end{split}
\end{equation}
We then observe that
\[
\begin{split}
\Big|\int u(y)\, [\rho_\varepsilon(x+h-y)&-\rho_\varepsilon(x-y)]\, dy\Big|=\left|\int_0^1\, \int u(y)\, \langle \nabla \rho_\varepsilon(x-y+s\,h),h\rangle\, dy\, ds\right|\\
&=\left|\int_0^1\, \int [u(y)-u(x+s\,h)]\, \langle \nabla \rho_\varepsilon(x-y+s\,h),h\rangle\, dy\, ds\right|\\
&\le \frac{\|\nabla \rho\|_\infty\, |h|}{\varepsilon^{N+1}}\,\int_0^1 \int_{B_{\varepsilon}(x+s\,h)\setminus B_{\frac{\varepsilon}{2}}(x+s\,h)}|u(y) -u(x+s\,h)|\, dy\\
&=\frac{\|\nabla \rho\|_\infty\, |h|}{\varepsilon^{N+1}}\,\int_0^1 \int_{B_{\varepsilon}(0)\setminus B_{\frac{\varepsilon}{2}}(0)}|u(x+z+s\,h) -u(x+s\,h)|\, dz\, ds,
\end{split}
\]
where in the second identity we used that $\int \nabla \rho_\varepsilon\, dx=0$. Finally by Jensen inequality and translation invariance of the $L^p$ norm, from \eqref{mostro} we can infer
\[
\begin{split}
\int_{\mathbb{R}^N} |u(x+h)-u(x)|^p\, dx&\le \frac{C\, |h|^p}{\varepsilon^{N+p}}\, \int_{B_{\varepsilon}(0)\setminus B_{\frac{\varepsilon}{2}}(0)}\int_{\mathbb{R}^N} |u(x+z) -u(x)|^p\, dx\,dz\\
&+\frac{C\,\|\rho\|_{\infty}}{\varepsilon^N}\,\int_{B_{\varepsilon}(0)\setminus B_{\frac{\varepsilon}{2}}(0)}\int_{\mathbb{R}^N} |u(x+z)-u(x)|^p\, dx\, dz.
\end{split}
\]
Since $\varepsilon< |h|$, the previous implies in particular
\[
\int_{\mathbb{R}^N} |u(x+h)-u(x)|^p\, dx\le \frac{C\, |h|^p}{\varepsilon^{N+p}}\, \int_{B_{\varepsilon}(0)\setminus B_{\frac{\varepsilon}{2}}(0)}\int_{\mathbb{R}^N} |u(x+z) -u(x)|^p\, dx\,dz,
\]
possibly with a different constant $C$, independent of $h$ and $\varepsilon$. If we now divide both sides by $|h|^{s\,p}$, we get
\begin{equation}
\label{immezzo}
\int_{\mathbb{R}^N} \frac{|u(x+h)-u(x)|^p}{|h|^{s\,p}}\, dx\le \frac{C\, |h|^{p(1-s)}}{\varepsilon^{N+p}}\, \int_{B_{\varepsilon}(0)\setminus B_{\frac{\varepsilon}{2}}(0)}\int_{\mathbb{R}^N} |u(x+z) -u(x)|^p\, dx\,dz.
\end{equation} 
If one is only interested in estimate \eqref{nikolski} with a constant independent of $s$, then at this point one can take $\varepsilon>|h|/2$, so that by construction
\[
\frac{|h|^{p(1-s)}}{\varepsilon^{N+p}}\le 2^{p(1-s)}\, \varepsilon^{-N-s\,p}\le 2^{p(1-s)}\, |z|^{-N-s\,p},\qquad z\in B_{\varepsilon}(0)\setminus B_\frac{\varepsilon}{2}(0),
\]
which inserted in \eqref{immezzo} would give
\[
\sup_{|h|>0}\int_{\mathbb{R}^N} \frac{|u(x+h)-u(x)|^p}{|h|^{s\,p}}\, dx\le C\,[u]^p_{W^{s,p}(\mathbb{R}^N)},
\]
with $C$ independent of $s$. 
\par
On the contrary, in order to get estimate \eqref{nikolski} displaying the sharp dependence on $s$, we proceed more carefully: we multiply both sides of \eqref{immezzo} by $\varepsilon^{p\,(1-s)-1}$ and integrate in $\varepsilon$ between $0$ and $|h|$. By further simplifying the common factor $|h|^{p\,(1-s)}$, this gives
\[
\frac{1}{p\,(1-s)}\int_{\mathbb{R}^N} \frac{|u(x+h)-u(x)|^p}{|h|^{s\,p}}\, dx\le C\int_0^{|h|}\frac{1}{\varepsilon^{N+p\,s+1}}\, \int_{B_{\varepsilon}(0)}\int_{\mathbb{R}^N} |u(x+z) -u(x)|^p\, dx\,dz.
\]
If we set for simplicity
\[
G(\varepsilon)=\int_{B_{\varepsilon}(0)}\int_{\mathbb{R}^N} |u(x+z) -u(x)|^p\, dx\,dz,\quad 0<\varepsilon\le |h|,
\]
by one-dimensional Hardy inequality we have\footnote{For $0<\tau\ll 1$, integrating by parts we have
\[
\begin{split}
\int_0^{|h|} \frac{(G(\varepsilon)-\tau)_+}{\varepsilon^{N+p\,s+1}}\, d\varepsilon&=-\frac{1}{N+p\,s} \frac{(G(|h|)-\tau)_+}{|h|^{N+p\,s}}+\frac{1}{N+p\,s}\int_{\{G(\varepsilon)>\tau\}} \frac{G'(\varepsilon)}{\varepsilon^{N+p\,s}}\, d\varepsilon\\
&\le \frac{1}{N+p\,s}\int_{0}^{|h|} \frac{G'(\varepsilon)}{\varepsilon^{N+p\,s}}\, d\varepsilon
\end{split}
\]
then we pass to the limit as $\tau$ goes to $0$.}
\[
\int_0^{|h|}\frac{1}{\varepsilon^{N+p\,s+1}}\, G(\varepsilon)\, d\varepsilon\le \frac{1}{N+p\,s}\, \int_0^{|h|} \frac{G'(\varepsilon)}{\varepsilon^{N+p\,s}}\, d\varepsilon,
\]
since $G(0)=0$ and $G$ is increasing. Then we observe that
\[
\begin{split}
\int_0^{|h|} \frac{G'(\varepsilon)}{\varepsilon^{N+p\,s}}\, d\varepsilon&=\int_0^{|h|}\frac{1}{\varepsilon^{N+p\,s}}\, \int_{\partial B_\varepsilon(0)}\int_{\mathbb{R}^N} |u(x+z) -u(x)|^p\, dx\,d\mathcal{H}^{N-1}(z)\, d\varepsilon\\
&=\int_{B_{|h|}(0)}\int_{\mathbb{R}^N} \frac{|u(x+z) -u(x)|^p}{|z|^{N+p\,s}}\, dx\,dz\le [u]^p_{W^{s,p}(\mathbb{R}^N)},
\end{split}
\]
which concludes the proof.
\end{proof}
\end{lm}
\begin{oss}
The previous result can be rephrased by saying that $W^{s,p}_0(\mathbb{R}^N)$ is continuously embedded in the relevant {\it Nikolskii space}. See for example \cite{Ad} for further details on this topic.
\end{oss}

\section{A remark on two different Sobolev spaces}

In order to avoid confusion, we point out that usually (see for example \cite{DPV}) the symbol $W^{s,p}_0(\Omega)$ denotes the closure of $C^\infty_0(\Omega)$ with respect to the norm
\[
u\mapsto [u]_{W^{s,p}(\Omega)}+\|u\|_{L^p(\Omega)}.
\]
In principle $W^{s,p}_0(\Omega)$ is larger than our $\widetilde W^{s,p}_0(\Omega)$ introduced in Section \ref{sec:sobolev}.
Indeed
\begin{equation}
\label{pezzinpiu}
[u]_{W^{s,p}(\mathbb{R}^N)}=[u]_{W^{s,p}(\Omega)}+2\,\int_{\Omega}\int_{\mathbb{R}^N\setminus\Omega} \frac{|u(x)|^p}{|x-y|^{N+s\,p}}\, dx\, dy,
\end{equation}
and there could exist $\Omega\subset\mathbb{R}^N$ and $u\in W^{s,p}_0(\Omega)$ such that the second integral on the right-hand side is infinite. Though we did not need this result in the paper, for completeness we record a sufficient condition for the two spaces to coincide.
\begin{prop}
Let $s\in(0,1)$ and $1<p<\infty$ be such that $s\,p \not=1$. Let $\Omega\subset\mathbb{R}^N$ be an open bounded Lipschitz set. Then there exists a constant $C=C(N,s,p,\Omega)>0$ such that
\[
[u]_{W^{s,p}(\mathbb{R}^N)}+\|u\|_{L^p(\Omega)}\le C\, \Big([u]_{W^{s,p}(\Omega)}+\|u\|_{L^p(\Omega)}\Big),\qquad \mbox{ for every } u\in C^\infty_0(\Omega).
\]
In particular \(\widetilde W^{s,p}_0(\Omega)=W^{s,p}_0(\Omega)\) as Banach spaces.
\end{prop}
\begin{proof}
Let $u\in C^\infty_0(\Omega)$, for every $x\in \Omega$ we set
\[
\delta_\Omega(x)=\inf_{y\in\mathbb{R}^N\setminus\Omega} |x-y|,
\]
i.e. this is the distance of $x$ from the complement of $\Omega$. Then we observe that 
$$
\mathbb{R}^N\setminus\Omega\subset \mathbb{R}^N\setminus B_{\delta_\Omega(x)}(x),\qquad x\in\Omega,
$$
which implies
\begin{equation}
\label{riscrivo}
\begin{split}
\int_\Omega \int_{\mathbb{R}^N\setminus \Omega} \frac{|u(x)|^p}{|x-y|^{N+s\,p}}\, dx\, dy&\le 
\int_\Omega \int_{\mathbb{R}^N\setminus B_{\delta_\Omega(x)}(x)} \frac{|u(x)|^p}{|x-y|^{N+s\,p}}\, dy\, dx\\
&=\int_\Omega  |u(x)|^p\, \left(N\,\omega_N\,\int_{\delta_{\Omega}(x)}^\infty \varrho^{-1-s\,p}\, d\varrho\right)\, dx\\
&=\frac{N\,\omega_N}{s\,p}\, \int_\Omega \frac{|u(x)|^p}{\delta_\Omega(x)^{s\,p}}\, dx.
\end{split}
\end{equation}
We now have to distinguish between $s\,p>1$ and $s\, p<1$: in the first case, by using in \eqref{riscrivo} the {\it fractional Hardy inequality} of \cite[Theorem 1.1]{Dy} 
\[
\int_\Omega \frac{|u(x)|^p}{\delta_\Omega(x)^{s\,p}}\, dx\le C\, [u]^p_{W^{s,p}(\Omega)},
\]
wtih $C=C(N,s,p,\Omega)>0$, we can conclude.
\par
In the case $s\,p<1$ the previous Hardy inequality can not hold true (see \cite[Section 2]{Dy}), but we have
\[
\int_\Omega \frac{|u(x)|^p}{\delta_\Omega(x)^{s\,p}}\, dx\le C'\,\left( [u]^p_{W^{s,p}(\Omega)}+\|u\|^p_{L^p(\Omega)}\right),\qquad u\in C^\infty_0(\Omega),
\]
with $C'=C'(N,s,p,\Omega)$, see \cite{Dy}.
\end{proof}

\begin{oss}
We point out that for the seminorm $[\,\cdot\,]_{W^{s,1}(\Omega)}$ with $\Omega\subset\mathbb{R}^N$ open bounded Lipschitz set, we have a result analogous to that of Proposition \ref{prop:conv} (see \cite[Theorem 1]{Da}). The case of a general open set $\Omega$ is slightly more complicated and can be found in \cite[Theorem 1.9]{LeS}. 
\end{oss}

\section{Pointwise inequalities}

\label{app:pointwise}

We collect some inequalities needed for the proof of the $L^\infty$ estimate for eigenfunctions.
\begin{lm}
Let $1<p<\infty$ and $\beta\ge 1$. For every $a,b\ge 0$ there holds
\begin{equation}
\label{puntuale}
|a-b|^{p-2}\, (a-b)\,(a^\beta-b^\beta)\ge \frac{\beta\, p^p}{(\beta+p-1)^p}\,\left|a^\frac{\beta+p-1}{p}-b^\frac{\beta+p-1}{p}\right|^p. 
\end{equation}
\end{lm}
\begin{proof}
We first observe that if $a=b$, then \eqref{puntuale} is trivially true. Let us then suppose that $a\not = b$ and of course we can suppose that $a>b$, without loss of generality. Then by collecting $a^{\beta+p-1}$ on both sides of \eqref{puntuale} and setting $t=b/a<1$, we get that \eqref{puntuale} is equivalent to
\[
(1-t)^{p-1}\,(1-t^\beta)\ge \frac{\beta\, p^p}{(\beta+p-1)^p}\,\left(1-t^\frac{\beta+p-1}{p}\right)^p,\qquad 0\le t<1.
\]
This inequality is just an easy consequence of Jensen inequality. Indeed, we have
\[
\begin{split}
(1-t)^{p-1}\, \frac{1-t^\beta}{\beta}=(1-t)^{p-1}\, \int_t^1 s^{\beta-1}\, ds\ge\left(\int_t^1 s^\frac{\beta-1}{p}\, ds\right)^p= \frac{p^p}{(\beta+p-1)^p}\left(1-t^\frac{\beta+p-1}{p}\right)^p,
\end{split}
\]
which gives the desired inequality.
\end{proof}
Actually, we used the following version of the previous result.
\begin{lm}
Let $1<p<\infty$ and $\beta\ge 1$. For every $a,b,M\ge 0$ there holds
\begin{equation}
\label{puntuale2}
|a-b|^{p-2}\, (a-b)\,(a^\beta_M-b^\beta_M)\ge \frac{\beta\, p^p}{(\beta+p-1)^p}\,\left|a^\frac{\beta+p-1}{p}_M-b^\frac{\beta+p-1}{p}_M\right|^p, 
\end{equation}
where we set
\[
a_M=\min\{a,M\}\qquad \mbox{ and }\qquad b_M=\min\{b,M\}.
\]
\end{lm}
\begin{proof}
By using inequality \eqref{puntuale} we get
\[
|a_M-b_M|^{p-2}\, (a_M-b_M)\,(a^\beta_M-b^\beta_M)\ge \frac{\beta\, p^p}{(\beta+p-1)^p}\,\left|a^\frac{\beta+p-1}{p}_M-b^\frac{\beta+p-1}{p}_M\right|^p.
\]
To conclude, we just need to prove that
\begin{equation}
\label{intermedio}
|a_M-b_M|^{p-2}\, (a_M-b_M)\,(a^\beta_M-b^\beta_M)\le |a-b|^{p-2}\, (a-b)\,(a^\beta_M-b^\beta_M).
\end{equation}
Let us suppose at first that $a\ge b$. Of course, if $a=b$ inequality \eqref{intermedio} is trivially satisfied, so we can consider $a>b$. In this case we have two possibility:
\[
b\ge M\qquad \mbox{ or }\qquad  b<M.
\]
In the first case, then $a_M=b_M=M$ and \eqref{intermedio} is satisfied. In the second case \eqref{intermedio} reduces to
\[
(a_M-b)^{p-1}\,(a^\beta_M-b^\beta)\le (a-b)^{p-1}\,(a^\beta_M-b^\beta),
\]
which is equivalent to
\[
a_M-b\le a-b.
\]
As the latter is trivially verified, the validity of \eqref{intermedio} is checked for $a\ge b$.
It is only left to observe that the discussion for the case $a\le b$ is exactly the same, so the proof in concluded.
\end{proof}

\end{document}